\documentclass[11pt,a4paper,reqno]{amsart}
\usepackage{amsmath,amscd,amstext,amsthm,amsfonts,latexsym,mathtools}
\usepackage[dvips]{graphicx}
\usepackage{enumitem}
\usepackage{bbm,xcolor}
\usepackage[colorlinks=true, linkcolor=blue, urlcolor=red, citecolor=blue, hyperindex, backref]{hyperref}

\newenvironment{claim}[1]{\par\noindent\underline{Claim:}\space#1}{}

\theoremstyle{plain}
\newtheorem{theorem}{Theorem}[section]
\newtheorem{proposition}[theorem]{Proposition}
\newtheorem{lemma}[theorem]{Lemma}
\newtheorem{corollary}[theorem]{Corollary}
\newtheorem{example}[theorem]{Example}
\newtheorem{maintheorem}{Theorem}

\theoremstyle{definition}
\newtheorem{remark}[theorem]{Remark}
\newtheorem{definition}[theorem]{Definition}

\newtheorem{question}[theorem]{Question}

\newcommand{\field}[1]{\mathbb{#1}}

\newcommand{\ZZ}{\field{Z}}
\newcommand{\NN}{\field{N}}

\newcommand{\ga} {\gamma}    
       
\newcommand{\ep} {\varepsilon}
\newcommand{\vep}{\varepsilon}

\newcommand{\la} {\lambda}      

\newcommand{\ro}{\rho}

\newcommand{\si} {\sigma}

\newcommand{\cK}{{\mathcal K}}
\newcommand{\cF}{{\mathcal F}}

\newcommand{\cP}{{\mathcal P}}

\newcommand{\cD}{{\mathcal D}}

\newcommand{\cE}{{\mathcal E}}

\newcommand{\cU}{{\mathcal U}}

\newcommand{\cN}{{\mathcal N}}

\newcommand{\diam}{\operatorname{{diam}}}

\newcommand{\supp}{\operatorname{{supp}}}

\newcommand{\mdim}{\mathrm{mdim}}

\begin{document}

\title[Quantization and the classical variational principle]{On quantization and the classical variational principle for the metric mean dimension}

\author[M. Carvalho]{Maria Carvalho}
\address{CMUP \& Departamento de Matem\'atica, Faculdade de Ci\^encias da Universidade do Porto, Rua do Campo Alegre 687, Porto, Portugal.}
\email{mpcarval@fc.up.pt}

\author[G. Pessil]{Gustavo Pessil}
\address{Institute of Mathematics of the Polish Academy of Sciences,  Śniadeckich 8, Warsaw,
00-656, Poland.}
\email{Gustavo.pessil@outlook.com}

%\thanks{* Corresponding author: }


\date{\today}

\begin{abstract}
This paper investigates the relationship between quantization of measures and metric mean dimension of topological dynamical systems. We introduce the concept of mean quantization dimension for invariant probability measures and establish a classical variational principle: the metric mean dimension of a dynamical system is equal to the maximum mean quantization dimension among all invariant measures. This approach effectively characterizes the complexity of systems with infinite entropy by identifying a measure that captures information across all scales; and yields a fundamental property that allows for the exchange of limits and suprema in the Lindenstrauss-Tsukamoto variational principles, a feat that most known entropy-like maps fail to achieve due to convexity. Nevertheless, we show that the Katok and Shapira entropies do satisfy this property and, therefore, a classical variational principle for the metric mean dimension, for which maximizing measures always exist.

%For a homeomorphism $T\colon X \to X$ on a compact metric space $(X,d)$, one can find in the literature several distinct notions of measure-theoretic $\vep$-entropy maps $F(\mu,\vep)$, defined on the space of Borel $T$-invariant probability measures of $X$ and  satisfying
%\begin{equation*}\label{eq:MTEE}
%\mathrm{\overline{mdim}_M}(X,d,T) \, =\, \limsup_{\varepsilon \, \to \, 0^+}\, \sup_{\mu \,\in\, \mathcal{P}_T(X)} \, \frac{F(\mu,\vep)}{\log\,(1/\varepsilon)}
%\end{equation*}
%where $\mathrm{\overline{mdim}_M}(X,d,T)$ stands for the upper metric mean dimension. A major question regards the change in the order in which $\limsup_{\varepsilon}$ and $\sup_{\mu \,\in\, \mathcal{P}_T(X)}$ appear in the previous equality. We introduce the notion of mean quantization dimension of a measure and prove that
%\begin{equation*}\label{eq:MTEE}
%\mathrm{\overline{mdim}_M}(X,d,T) \,\, =\, \max_{\mu \,\in\, \mathcal{P}_T(X)} \, \mathrm{\overline{mdim}_Q}(\mu). 
%\end{equation*}
%This concept exhibits a fundamental property that yields a sufficient condition on the map $F$ under which the aforementioned change in order is valid. \\

%\color{blue}
%We introduce the notion of dynamical quantization number $Q(\mu, \ep)$ of a measure $\mu\in\cP(X)$ at scale $\ep>0$, and prove that
%\begin{equation*}\label{eq:MTEE}
%\mathrm{\overline{mdim}_M}(X,d,T) \,\, =\, \max_{\mu \,\in\, \mathcal{P}_T(X)} \, \limsup_{\ep \, \to \, 0^+} \, \frac{Q(\mu,\vep)}{\log\,(1/\varepsilon)}. 
%\end{equation*}
%\color{black}

\end{abstract}

\maketitle

\tiny
%\tableofcontents
\normalsize

%%%%%%%%%%%%%%%%%%%%%%%%%%%%%%%%%%%%%%%%%%%%%%%%%%%%%%%%%%%%%%%%%%
\section{Introduction}

Let $(X,d)$ be a compact metric space and $T\colon X \to X$ be a homeomorphism. Denote by $\mathcal{P}(X)$ the set of Borel probability measures on $X$ with the weak$^*$-topology; by $\mathcal{P}_T(X)$ its subset of $T$-invariant measures; by $\mathcal{E}_T(X) \subset \mathcal{P}_T(X)$ the set of ergodic elements; and, given $\mu \in \mathcal{P}(X)$, $\supp(\mu)$ stands for its support. 

\smallskip

Measure-theoretic and topological entropy are classical invariants in the theory of dynamical systems and are related by the classical variational principle 
\begin{equation}\label{Classical_VP_Entropy}
    h_{top}(T) \,\,\, = \, \sup_{\mu \, \in \, \cP_T(X)}\, h_{\mu}(T) \,\, \,= \, \sup_{\mu \, \in \, \cE_T(X)}\, h_{\mu}(T)
\end{equation}
which sometimes provides a natural way to choose important elements from $\cP_T(X)$. Entropy-dimensional concepts have been recently developed in order to estimate the complexity of infinite entropy systems from a better suited viewpoint. The \emph{metric mean dimension}, introduced by E. Lindenstrauss and B. Weiss in \cite{LW2000}, is both a dynamical analog of the box dimension and a dimensional analog of the topological entropy. In what follows, we denote the upper and lower metric mean dimensions by $\mathrm{\overline{mdim}_M}(X,d,T)$ and $\mathrm{\underline{mdim}_M}(X,d,T)$, respectively, thus emphasizing their dependence on the fixed metric $d$ on the space $X$ where the dynamics $T$ acts. These concepts vanish if the topological entropy of $T$ is finite; if, otherwise, $T$ has infinite entropy, they quantify the speed at which the entropy at scale $\varepsilon$ approaches $+\infty$ as the scale goes to zero. The choice of the metric $d$ has impact precisely on such convergence rate.

In \cite{LT18}, E. Lindenstrauss and M. Tsukamoto established a connection between the upper metric mean dimension and the rate distortion function $\mathcal{R}_{\mu}$.
%, where $\mu \in \mathcal{P}_T(X)$. 
More precisely, they showed that with the tame growth of covering numbers
\begin{equation*}\label{eq:vp0}
\mathrm{\overline{mdim}_M}(X,d,T) \, \,=\,\, \limsup_{\varepsilon \, \to \, 0^+}\, \frac{\sup_{\mu \,\in\, \mathcal{P}_T(X)} \, {\mathcal{R}}_{\mu}(\varepsilon)}{\log\,(1/\varepsilon)}.
\end{equation*}
Meanwhile, in \cite{GS}, Y. Gutman and A. \'Spiewak proved a corresponding relation between the upper metric mean dimension and the Rényi entropy, namely
\begin{equation*}\label{eq:vp2}
\mathrm{\overline{mdim}_M}(X,d,T) \, = \, \limsup_{\varepsilon \, \to \, 0^+}\,\frac{\sup_{\mu \,\in\, \mathcal{E}_T(X)} \,\, \inf_{\mathrm{diam}(P)\, \leq \, \varepsilon}\, h_\mu(P)}{\log\,(1/\varepsilon)}\,
\end{equation*}
where the infimum is taken over the finite Borel partitions $P$ of $X$ with diameter at most $\varepsilon$ and $h_\mu(P)$ stands for the 
entropy of $\mu$ with respect to $P$. Moreover, Problem $3$ in \cite{GS} asked whether the upper metric mean dimension could be expressed in terms of the Brin-Katok local entropy. An affirmative answer to this problem was given by R. Shi in \cite{SHI}:
\begin{equation*}\label{eq:vp3}
\mathrm{\overline{mdim}_M}(X,d,T) \, =\, \limsup_{\varepsilon \, \to \, 0^+}\,\frac{\sup_{\mu \,\in\, \mathcal{E}_T(X)} \,\overline{h}^{BK}_\mu(\varepsilon)}{\log\,(1/\varepsilon)}\,.
\end{equation*}
\smallskip

Apart from the ones mentioned above, many other distinct notions of measure-theoretic metric mean dimension were shown to exist and be inspired by some version of entropy, such as Katok metric entropy (cf. \cite{SHI}), Shapira entropy (cf. \cite{SHI}) and Pfister-Sullivan entropy (cf. \cite{Wang}), just to mention a few. A unified framework was provided by E. Chen, R. Yang and X. Zhou in \cite{EYZ2025}, where they considered the following family of entropy-like functions of a measure $\mu$ at a given scale $\ep$ (see the precise definitions in \cite{EYZ2025} and Section~\ref{se:definition}):
\begin{equation}\label{def:Dcal}
    \cD\,\,=\,\, \Bigg\{
\begin{aligned}
&\underline{h}_\mu^K(\ep,\delta),\,\,\,
\overline{h}_\mu^K(\ep,\delta),\,\,\,
\underline{h}_\mu^K(\ep),\,\,\,
\overline{h}_\mu^K(\ep),\,\,\,
\overline{h}_\mu^{BK}(\ep),\,\,\, 
h_\mu^S(\ep,\delta),\,\,
\\
&
PS_\mu(T,\ep),\,\,\,
{\mathcal{R}}_{\mu,\,L^\infty}(\ep),\,\,\,
\inf_{\mathrm{diam}(P)\,\le\, \varepsilon} h_\mu(P),\,\,\,
%\inf_{\mathrm{diam}(\cU)\,\le\, \varepsilon} h^S_\mu(\cU,\delta)
\cP_\mu(T,d,\ep)
\end{aligned}
\Bigg\}.
\end{equation} 
The authors show that, for any ergodic probability measure $\mu\in\cE_T(X)$, the limits 
\begin{equation}\label{Flimit}
 \limsup_{\varepsilon \, \to \, 0^+}\,\frac{F(\mu,\varepsilon)}{\log\,(1/\varepsilon)}\qquad\text{and}\qquad \liminf_{\varepsilon \, \to \, 0^+}\,\frac{F(\mu,\varepsilon)}{\log\,(1/\varepsilon)} 
\end{equation} 
do not depend on the choice of $F\in\cD$ (so they coincide with the upper and lower $L^\infty$-rate distortion dimensions, respectively) and that, for every $F\in\cD$,
\begin{equation}\label{vpErcai}
    \mathrm{\overline{mdim}_M}(X,d,T) \,=\, \limsup_{\varepsilon \, \to \, 0^+}\,\frac{\sup_{\mu \,\in\, \mathcal{P}_T(X)} \,F(\mu,\varepsilon)}{\log\,(1/\varepsilon)} \,=\, \limsup_{\varepsilon \, \to \, 0^+}\,\frac{\sup_{\mu \,\in\, \mathcal{E}_T(X)} \,F(\mu,\varepsilon)}{\log\,(1/\varepsilon)}.
\end{equation} 
The corresponding equalities also hold for the lower limits. We stress that, in equality \eqref{vpErcai}, maximization over the space of probability measures is done for a fixed scale, which only afterwards is made to go to zero. A recurring question in the literature 
asks whether we may exchange the order of $\limsup_\ep$ and $\sup_\mu$ in \eqref{vpErcai}. An affirmative answer to this question would suggest natural definitions of measure-theoretic metric mean dimension, namely 
$$\mathrm{\overline{mdim}_M}(X,d,T,F)(\mu) \, = \, \limsup_{\varepsilon \, \to \, 0^+}\,\frac{F(\mu, \vep)}{\log\,(1/\varepsilon)}$$
and provide a classical variational principle for the metric mean dimension in terms of $F$; moreover, if $\sup_\mu$ could be replaced by $\max_\mu$, this principle would single out a measure which captures the complexities of the system on all scales $\vep$. It is known that this order can be changed under the marker property (cf. \cite{LT19}), but there are relevant systems that do not satisfy this condition, as shown in \cite{Shi23, TTY}. The following examples shed some light on this problem. 

\smallskip

\begin{example}\cite[Section VIII]{LT18}\label{ex:RT2} \emph{Consider the set $A = \{0\}\cup \{1/n \colon n \in \NN\}$, $X = A^\ZZ$ endowed with the metric
$$d(x,y) \, = \, \sum_{n \, \in \, \ZZ}\, 2^{-|n|} \, |x_n - y_n|$$
and $T \colon X \to X$ the standard shift. Then $\mathrm{\overline{mdim}_M}(X,d,T)= 1/2$ (cf. \cite{VV}) and 
$$\forall \, \mu \in \cP_T(X) \quad \quad \lim_{\vep \, \to \, 0^+}\, \frac{\mathcal{R}_\mu(\vep)}{\log \, (1/\vep)} \, = \, 0.$$
So,
$$\mathrm{\overline{mdim}_M}(X,d,T)\,\, >\, \sup_{\mu \, \in \, \mathcal{P}_T(X)} \,\limsup_{\vep \, \to \, 0^+} \, \frac{\mathcal{R}_\mu(\vep)}{\log \, (1/\vep)}.$$}
\end{example}

\smallskip

\begin{example}\cite[Theorem 1.4]{Tsu2025}\label{ex:RT} \emph{In \cite{Tsu2025} Tsukamoto constructs a homeomorphism $T$ on a compact metric space $(X,d)$ with the following properties:}

\noindent \emph{$\bullet$ For any ergodic measure $\nu \in \cP_T(X)$ we have $\,\displaystyle \limsup_{\vep \, \to \, 0^+} \, \frac{\mathcal{R}_\nu(\vep)}{\log \, (1/\vep)}\, =\, 0.$}

\noindent \emph{$\bullet$ There exists $\eta \in \cP_T(X)$ such that 
$\,\displaystyle \limsup_{\vep \, \to \, 0^+} \, \frac{\mathcal{R}_\eta(\vep)}{\log \, (1/\vep)}\, =\, +\infty.$}\\

\noindent \emph{Consequently, one has $\mathrm{\overline{mdim}_M}(X,d,T) = +\infty$ and
$$\mathrm{\overline{mdim}_M}(X,d,T)\,\, > \,  \sup_{\mu \, \in \, \mathcal{E}_T(X)} \,\limsup_{\vep \, \to \, 0^+} \, \frac{\mathcal{R}_\mu(\vep)}{\log \, (1/\vep)}.$$
Yet, it is still true for this example that 
\begin{equation}\label{ExTsukamoto}
    \mathrm{\overline{mdim}_M}(X,d,T)\,\, =\, \sup_{\mu \, \in \, \mathcal{P}_T(X)} \,\limsup_{\vep \, \to \, 0^+} \, \frac{\mathcal{R}_\mu(\vep)}{\log \, (1/\vep)}.
\end{equation}}
\end{example}

\smallskip

Why do they behave differently with respect to the change in the order in which $\limsup_\ep$ and $\sup_{\mu\,\in\,\cP_T(X)}$ appear? It turns out that a major role in this issue is played by the behavior of the measure-theoretic $\vep$-entropy under ergodic decompositions. On the one hand, if the upper metric mean dimension is finite, then the upper rate distortion dimension
$$\overline{R}(X,d,T,\mu)\, = \, \limsup_{\vep \, \to \, 0^+}\,  \frac{\mathcal{R}_{\mu}(\varepsilon)}{\log\,(1/\varepsilon)}$$
is a convex function of $\mu \in \mathcal{P}_T(X)$ with respect to its ergodic decomposition (cf. \cite[Theorem~1.1]{Tsu2025}). Thus, in Example~\ref{ex:RT2}, the change in order between $\limsup_\ep$ and $\sup_{\mu\,\in\,\cP_T(X)}$ produces a strict inequality. On the other hand, the system described in Example~\ref{ex:RT} has $\mathrm{\overline{mdim}_M}(X,d,T) = +\infty$ and the convexity no longer holds. This is the reason why \eqref{ExTsukamoto} is possible.
\smallskip

The purpose of this paper is to explain that we can understand this issue more clearly by comparing the maps $F \in \cD$ with a new concept that we introduce in the next section.

\section{Main results}

After recalling the problem of quantization of measures, we motivate the questions we address and describe our main contributions.

\subsection{Quantization}\label{ex:Q}

Let $(X,d)$ be a compact metric space. Quantization of measures, as addressed in \cite{GL}, consists in approximating, at arbitrarily small resolutions, a given measure by another with finite support. Let $D \in \{W_p\colon p \in [1,+\infty)\}\cup \{LP\}$ stand for either a Wasserstein distance or the L\'evy-Prokhorov distance on the space of Borel probability measures $\cP(X)$ (see the precise definitions in Subsection~\ref{sse:LP}).

\begin{definition}\label{def:qnumber}
\emph{Given $\mu\in\cP(X)$ and $\vep > 0$, the \emph{quantization number $Q_{\mu,D}(\ep)$ of $\mu$ with respect to $D$ at a scale $\vep$} is the least positive integer $N$ such that there exists $\nu\in\cP(X)$ with $D(\mu, \nu) < \vep$ and $\nu$ is supported on a set of cardinality $N$. The \emph{upper} and \emph{lower quantization dimensions of $\mu$} are defined as $$\overline{\dim}_Q(\mu,D)\,=\,\limsup_{\ep\,\to\,0^+}\,\frac{\log Q_{\mu,D}(\ep)}{\log(1/\epsilon)}\qquad\text{and}\qquad\underline{\dim}_Q(\mu,D)\,=\,\liminf_{\ep\,\to\,0^+}\,\frac{\log Q_{\mu,D}(\ep)}{\log(1/\epsilon)}.$$}
\end{definition}
\smallskip
It is known (cf. \cite[Proposition~3.4]{BB}) that the quantization number of $\mu$ at scale $\ep$ is bounded from above by the $\ep$-covering number of the space (see Subsection~\ref{sse:boxdim}). Moreover, the quantization dimensions are related to the upper and lower box-counting dimensions of the space by the following variational principles (see \cite[Proof of Theorem~3.9]{BB}): 
\begin{equation}\label{dimB_VP}
\overline{\dim}_B(X,d)\,\,=\,\max_{\mu\,\in\,\cP(X)}\,\overline{\dim}_Q(\mu,D)\qquad\text{and}\qquad \underline{\dim}_B(X,d)\,\,=\,\max_{\mu\,\in\,\cP(X)}\,\,\underline{\dim}_Q(\mu,D).
\end{equation}

\smallskip

\subsection{Dynamical quantization}\label{sse:meanQ}

We now address the corresponding dynamical objects. Let $T\colon X\to X$ be a continuous map. Given $D \in \{W_p\colon p \in [1,+\infty)\}\cup \{LP\}$ and $n \in \mathbb{N}$, we denote by $D_n$ the Wasserstein or the L\'evy-Prokhorov distance induced in $\cP(X)$ by the Bowen metric $d_n$ in $X$ (defined in \eqref{def:Bdn}).
\smallskip

\begin{definition} 
\emph{Given $D \in \{W_p\colon p \in [1,+\infty)\}\cup 
\{LP\}$, $\mu\in\cP(X)$ and $\vep > 0$, the \emph{upper} and \emph{lower dynamical quantization numbers of $\mu\in\cP(X)$ at scale $\ep$} are given by 
$$\overline{Q}_\mu(T,D,\ep)\,=\,\limsup_{n\,\to\,+\infty}\,\frac{1}{n}
\log Q_{\mu,D_n}(\ep)
\quad\text{and}\quad
\underline{Q}_\mu(T,D,\ep)\,=\,\liminf_{n\,\to\,+\infty}\,\frac{1}{n}\log Q_{\mu,D_n}(\ep).$$}
\end{definition}
\smallskip
The following relation between the metric entropy and the dynamical quantization numbers of an ergodic measure was proved in \cite{BB}.
\smallskip
\begin{theorem}\cite[Appendix A.3]{BB}\label{teo:BB}
Let $(X,d)$ be a compact metric space and $T\colon X\to X$ be a continuous map. Given $D \in \{W_p\colon p \in [1,+\infty)\}\cup \{LP\}$ and $\mu \in \cE_T(X)$, then
$$h_\mu(T) \,\,=\,\, \lim_{\ep\,\to\,0^+}\,\overline{Q}_\mu(T,D,\ep)\,\,=\,\, \lim_{\ep\,\to\,0^+}\,\,\underline{Q}_\mu(T,D,\ep).$$
\end{theorem}

\smallskip

\noindent We observe that, applying \cite[Proposition~3.4 \& (1.1)]{BB} to the space $(X,d_n)$, we obtain 
\begin{equation}\label{QuantLeqEnt}
\overline{Q}_\mu(T,D,\ep)\,\,\,\leq\,\,\, h_\ep(X,d,T)\qquad \forall\,\mu\,\in\,\cP(X),
\end{equation}
where 
$$h_\ep(X,d,T) \, = \, \limsup_{n \, \to \, +\infty}\, \frac{1}{n}\, \log \,S(X,n,\vep)$$
and $S(X,n,\vep)$ is the maximal cardinality of the $\vep$-separated subsets of $X$ (see Subsection~\ref{sse:sepspacov}). Combining \eqref{QuantLeqEnt} with Theorem~\ref{teo:BB} and the classical variational principle \eqref{Classical_VP_Entropy}, we get 
\begin{equation}\label{classicvp}
\begin{alignedat}{3}
h_{top}(T)\,\,\,
&= \,\,\sup_{\mu \,\in\, \cE_T(X)}\,\lim_{\ep\,\to\,0^+}\, \overline{Q}_\mu(T,D,\ep)\,\,
&= \,\,\sup_{\mu \,\in \,\cE_T(X)}\,\lim_{\ep\,\to\,0^+}\, \underline{Q}_\mu(T,D,\ep) \\
&= \,\,\sup_{\mu \,\in \,\cP_T(X)}\,\lim_{\ep\,\to\,0^+}\, \overline{Q}_\mu(T,D,\ep)\,\,\,
&= \,\,\sup_{\mu \,\in \,\cP_T(X)}\,\lim_{\ep\,\to\,0^+}\, \underline{Q}_\mu(T,D,\ep)\mathrlap{.}
\end{alignedat}
\end{equation}

\smallskip
\begin{definition}
\emph{The \emph{upper} and \emph{lower mean quantization dimensions of a measure $\mu\in\cP_T(X)$ with respect to the metric $D \in \{W_p\colon p \in [1,+\infty)\}\cup \{LP\}$} are given, respectively, by 
\begin{eqnarray*}
\overline{\mathrm{mdim}}_Q(X,d,T,D,\mu) &\,=\,& \limsup_{\ep\,\to\,0^+}\,\frac{\overline{Q}_\mu(T,D,\ep)}{\log\,(1/\ep)} \\ 
\underline{\mathrm{mdim}}_Q(X,d,T,D,\mu) &\,=\,&\liminf_{\ep\,\to\,0^+}\,\frac{\overline{Q}_\mu(T,D,\ep)}{\log\,(1/\ep)}.
\end{eqnarray*}
\noindent If they agree, the common value is denoted by $\mathrm{mdim}_Q(X,d,T,D,\mu)$ and called mean quantization dimension.}
\end{definition}

We will show (cf. Lemmas~\ref{QuantGrowthRate} and \ref{le:QK}) that, if $\mu$ is ergodic, then the mean quantization dimensions of $\mu$ do not change if one replaces $\overline{Q}_\mu(T,D,\ep)$ by $\underline{Q}_\mu(T,D,\ep)$ in their definitions. Also, by \eqref{QuantLeqEnt}, we have 
\begin{equation}\label{MQuantLeqMDim}
\overline{\mathrm{mdim}}_Q(X,d,T,D,\mu)\,\,\,\leq\,\,\,\overline{\mathrm{mdim}}_M(X,d,T)\qquad\forall \,\mu\,\in\,\cP_T(X).
\end{equation}

Our first result states that the dynamical quantization numbers can be included in the set $\cD$, that is, they satisfy the properties \eqref{Flimit} and \eqref{vpErcai}. Moreover, the mean quantization dimension satisfies a classical variational principle, where maximization is taken over the space of invariant probability measures, which can be seen both as a dimensional analog of \eqref{Classical_VP_Entropy} and as a dynamical analog of \eqref{dimB_VP}. Furthermore, measures with maximal mean quantization dimension always exist.

\begin{maintheorem}\label{teo:QD}
Let $(X,d)$ be a compact metric space, $T\colon X\to X$ be a homeomorphism and $D \in \{W_p\colon p \in [1,+\infty)\}\cup \{LP\}$. Then: 

\medskip

\begin{itemize}
\item[$(a)$] For every $\mu\,\in\,\cE_T(X)$  and $F\,\in\,\cD$,
\begin{eqnarray*}
\overline{\mathrm{mdim}}_Q(X,d,T,D,\mu) &\,=\,& \limsup_{\ep\,\to\,0^+}\,\frac{F(\mu,\ep)}{\log\,(1/\ep)} %\\ 
%\underline{\mathrm{mdim}}_Q(X,d,T,D,\mu) &\,=\,&\liminf_{\ep\,\to\,0^+}\,\frac{F(\mu,\ep)}{\log\,(1/\ep)}.
\end{eqnarray*}
\item[$(b)$] \emph{[Variational principle I]}
\begin{equation*}
    \begin{alignedat}{3}
        \overline{\mdim}_M(X,d,T) \,&=\,\, \limsup_{\varepsilon \, \to \, 0^+}\, \sup_{\mu \,\in\, \mathcal{E}_T(X)} \, \frac{\overline{Q}_\mu(T,D,\ep)}{\log\,(1/\varepsilon)} \, \,\,&=\,\,\,\, \limsup_{\varepsilon \, \to \, 0^+}\, \sup_{\mu \,\in\, \mathcal{E}_T(X)} \, \frac{\underline{Q}_\mu(T,D,\ep)}{\log\,(1/\varepsilon)}
        \\
        &=\,\, \limsup_{\varepsilon \, \to \, 0^+}\, \sup_{\mu \,\in\, \mathcal{P}_T(X)} \, \frac{\overline{Q}_\mu(T,D,\ep)}{\log\,(1/\varepsilon)} \,\, \,&=\,\,\, \limsup_{\varepsilon \, \to \, 0^+}\, \sup_{\mu \,\in\, \mathcal{P}_T(X)} \, \frac{\underline{Q}_\mu(T,D,\ep)}{\log\,(1/\varepsilon)}\mathrlap{.}
    \end{alignedat}
\end{equation*}
\smallskip

\item[$(c)$] \emph{[Variational principle II]}

\begin{equation}\label{eq:vpQ}
\begin{alignedat}{3}
\mathrm{\overline{mdim}_M}(X,d,T)\, \,
&=\,\, \max_{\mu\,\in\,\cP_T(X)}\,\, \,\mathrm{\overline{mdim}_Q}(X,d,T,D,\mu) \\
& =\,\, \max_{\mu\,\in\,\cP_T(X)}\, \,\limsup_{\ep\,\to\,0^+}\,\frac{\overline{Q}_\mu(T,D,\ep)}{\log\,(1/\ep)} \\
& =\,\, \max_{\mu\,\in\,\cP_T(X)}\, \,\limsup_{\ep\,\to\,0^+}\,\frac{\underline{Q}_\mu(T,D,\ep)}{\log\,(1/\ep)} \mathrlap{.}
\end{alignedat}
\end{equation}
\smallskip 

\item[$(d)$] Unless $T$ is uniquely ergodic, there exist uncountably many probabilities measures with maximal mean quantization dimension.
\end{itemize}
\smallskip

\color{black}
The corresponding statements for the lower metric mean dimension also hold. 
\end{maintheorem}

\smallskip

We observe that, in general, the maximum in \eqref{eq:vpQ} cannot be replaced by the supremum over the space of ergodic measures. Indeed, in Section~\ref{se:proofD} we discuss an example of a compact metric space $(X_0,d_0)$ and a homeomorphism $T_0\colon X_0 \to X_0$ for which
$$\mathrm{\overline{mdim}_M}(X_0,d_0,T_0) \,\, >\, \sup_{\mu\,\in\,\cP_{T_0}(X_0)}\,\,  \limsup_{\vep \, \to \, 0^+} \, \,\frac{\overline{h}_\mu^{BK}(\ep)}{\log \,(1/\vep)}.$$
By Theorem~\ref{teo:QD}$(a)$, for any $F \in \cD$ and $D \in \{W_p\colon p \in [1,+\infty)\}\cup \{LP\}$ we have
$$\mathrm{\overline{mdim}_M}(X_0,d_0,T_0)\,\,>\,\, \sup_{\mu\,\in\,\cE_{T_0}(X_0)}\,\mathrm{\overline{mdim}_Q}(X_0,d_0,T_0,D,\mu)\,\,=\,\, \sup_{\mu\,\in\,\cE_{T_0}(X_0)}\,\,\limsup_{\ep\,\to\,0^+}\,\frac{F(\mu, \ep)}{\log\,(1/\ep)}.$$

\smallskip

%%%%%%%%%%%%%%%%%%%%%%%%%%%%%%%%%%%%%%%%%%%%%%%%%%%%%%%%%%%%

\subsection{Main property}

The fundamental aspect that enables us to prove \eqref{eq:vpQ} is that the quantization numbers of an invariant measure are, up to normalization of scales by the weights appearing in its ergodic decomposition, at least the quantization numbers of its ergodic components. We express this idea using the following property.

\begin{definition}\label{def:MP}
\emph{Let $\cK$ be a convex subset of $\cP(X)$. A map $F\colon \cK \times (0,+\infty)\, \to  [0, +\infty]$ is said to satisfy the \emph{scaled monotonicity under measure domination} (s.m.m.d. for short) in $\cK$ if there exists a constant $C>0$ such that, for every $\mu, \nu \in \cK$ with $\mu \geq t\,\nu$ for some $t \in (0,1)$, one has
$$F(\mu, t\,\vep) \, \geq\, F(\nu, C\,\vep) \quad \forall\, \vep>0.$$}
\end{definition}

\smallskip

We observe that, if $\cK = \cP_T(X)$ and $\mu \in \cK$ is ergodic, then the condition $\mu \geq t\,\nu$ for some $\nu \in \cK \setminus \{\mu\}$ is empty. Otherwise, $\nu$ would be absolutely continuous with respect to $\mu$, hence $\nu$ is ergodic, and, as $\nu$ is $T$-invariant, $\nu = \mu$.\\

\smallskip

The s.m.m.d.~property is essentially incompatible with convexity. For example, assume that map $F(\cdot,\ep)$ is convex in $\cK$ for every $\ep>0$ and satisfies the s.m.m.d.~property. Then, for any $\mu, \nu \in \cK$, we have $$F(\nu,C\ep)\,\,\leq\,\,F\big((\mu+\nu)/2,\ep/2\big)\,\,\leq\,\,
\big[F(\mu,\ep/2)\,+\,F(\nu,\ep/2)\big]/2.$$ 
Hence,
$$\limsup_{\vep \, \to \, 0^+} \,\frac{F(\nu,\ep)}{\log(1/\ep)}\,\,\leq\,\,\limsup_{\vep \, \to \, 0^+} \,\frac{F(\mu,\ep)}{\log(1/\ep)}$$
and we conclude that the map 
$$\mu\,\mapsto\,\limsup_{\vep \, \to \, 0^+} \,\frac{F(\mu,\ep)}{\log(1/\ep)}$$ is constant in $\cK$. Moreover, the s.m.m.d.~property implies that, whenever $\mu= t\,\mu_1\,+\,(1-t)\,\mu_2$ for some $t \in (0,1)$ and $\mu_1, \mu_2 \in \cP(X)$, then  \begin{equation}\label{StrongConcavity}
    \limsup_{\vep \, \to \, 0^+} \,\frac{F(\mu,\ep)}{\log(1/\ep)}\,\,\geq\,\,\max \Bigg\{\limsup_{\vep \, \to \, 0^+} \,\frac{F(\mu_1,\ep)}{\log(1/\ep)},\,\,\limsup_{\vep \, \to \, 0^+} \,\frac{F(\mu_2,\ep)}{\log(1/\ep)}\Bigg\}
\end{equation} which is a rather strong form of concavity.

\smallskip

\begin{maintheorem}\label{teo:VPg}
Let $(X,d)$ be a compact metric space and $\cK$ be a convex subset of $\cP(X)$. Assume that $F \colon \cK \times (0, +\infty)\, \to [0, +\infty]$ satisfies the s.m.m.d.~property in $\cK$. Then:
\smallskip

\noindent $(a)$
\begin{equation}\label{eq:vpgeralg}
\limsup_{\vep \, \to \, 0^+} \, \frac{\sup_{\mu \, \in \, \cK}\,F(\mu, \vep)}{\log \, (1/\vep)}
 \,\, =\,\, \max_{\mu\,\in\,\cK}\,\,  \limsup_{\vep \, \to \, 0^+} \, \,\frac{F(\mu, \vep)}{\log \,(1/\vep)}.    
 \end{equation}
\smallskip

\noindent $(b)$ If, in addition, $F$ is non-increasing in the variable $\ep$, then 
 \begin{equation}\label{eq:vpnon}
\liminf_{\vep \, \to \, 0^+} \, \frac{\sup_{\mu \, \in \, \cK}\,F(\mu, \vep)}{\log \, (1/\vep)}
 \,\, =\,\, \max_{\mu\,\in\,\cK}\,\,  \liminf_{\vep \, \to \, 0^+} \, \,\frac{F(\mu, \vep)}{\log \,(1/\vep)}.   
\end{equation}
\smallskip

\noindent $(c)$ The space $\cK_{max}\subset\cK$ of maximizing measures in \eqref{eq:vpgeralg} (resp. \eqref{eq:vpnon}) satisfies the following property: $$\mu\,\in\,\cK_{max}, \,\,\,\eta\,\in\,\cK,\,\,\,\la\,\in\,[0,1) \quad \implies\quad (1-\la)\,\mu\,+\,\la\,\nu\,\in\, \cK_{max}.$$
\end{maintheorem}

\medskip

We note that if the map $\mu \in \cK \mapsto \limsup_{\vep \, \to \, 0^+} \,\frac{F(\mu,\ep)}{\log(1/\ep)}$ is concave, then the space $\cK_{max}$ is convex. Item $(c)$ above states that under the s.m.m.d.~property, if $\cK$ is not a singleton, then the set $\cK_{max}$ is huge. 

\smallskip

The previous result provides a generalization of item (c) of Theorem~\ref{teo:QD} that also answers the next question.

\begin{question} \emph{Under what assumptions on $F\colon \cP_T(X) \times (0,+\infty) \,\to  [0, +\infty]$ can we ensure that
$$\mathrm{\overline{mdim}_M}(X,d,T) \, = \, \sup_{\mu \,\in\, \mathcal{P}_T(X)} \,\limsup_{\varepsilon \, \to \, 0^+}\,\frac{F(\mu, \vep)}{\log\,(1/\varepsilon)}\,?$$}
\end{question}

\begin{corollary}\label{teo:VP}
Let $(X,d)$ be a compact metric space and $T\colon X\to X$ be a homeomorphism. If $F \in\cD$ satisfies the s.m.m.d.~property in $\cP_T(X)$, then
\begin{equation}\label{eq:vpgeral}
\mathrm{\overline{mdim}_M}(X,d,T)
 \,\, =\,\, \max_{\mu\,\in\,\cP_T(X)}\,\,  \limsup_{\vep \, \to \, 0^+} \, \,\frac{F(\mu, \vep)}{\log \,(1/\vep)}. 
\end{equation}
Since $F \in \mathcal{D}$ is non-increasing in the variable $\vep$, the corresponding lower version of \eqref{eq:vpgeral} also holds. Moreover, unless $T$ is uniquely ergodic, there exist uncountably many probabilities measures with maximal metric mean dimension.
\end{corollary}

%\color{red}
%Medidas ergódicas distintas são sempre singulares. Logo a condição $\mu \geq t \nu$ nunca ocorre de qualquer jeito e podemos tirar o comentário abaixo.

% OK.

%We note that, for the previous result to be valid, it is not enough to assume that the s.m.m.d.~property is satisfied in $\cE_T(P)$. Actually, the upper Brin-Katok $\vep$-entropy satisfies the s.m.m.d.~property in $\cE_T(P)$ (due to the inequalities in \cite[Lemma~3.1]{EYZ2025}, Lemma~\ref{QuantGrowthRate} and Lemma~\ref{def:MPQ}), but we show in Subsection~\ref{BKcounterex} that it does not comply with the variational principle \eqref{eq:vpgeral}.

%\color{black}

%%%%%%%%%%%%%%%%%%%%%%%%%%%%%%%%%%%%%%%%%%%%%%%%%%%%%%%%%%%

\subsection{Variational principles for maps in $\cD$}
Since all maps $F\in\cD$, as well as the dynamical quantization numbers, have the same growth rate at ergodic measures, it is worthwhile to investigate how each individual $F$ acts on invariant measures, as this is the main key to obtaining the variational principle \eqref{eq:vpQ}. In spite of the common probabilistic and geometric hybrid nature of these concepts, some of them satisfy the s.m.m.d.~property (as happens with Katok $\vep$-entropy and Shapira $\vep$-entropy; cf. Remark~\ref{rem:smmd}), while others are convex (like Brin-Katok $\vep$-entropy; cf. Lemma~\ref{BK_convex}). This is the content of the next result.

\begin{maintheorem}\label{teo:EYZ}
The following assertions hold: 
\begin{itemize}
\item[$(a)$] Let $(X,d)$ be compact metric space and $T \colon X \to X$ be a homeomorphism. Then, for every $\mu\in\cP_T(X)$,
\begin{eqnarray*}
\overline{\mathrm{mdim}}_Q(X,d,T,LP,\mu) &=& %\limsup_{\ep\,\to\,0^+}\,\frac{\overline{Q}_\mu(T,d,\ep)}{\log\,(1/\ep)}\,\,\,= \,\,\,
\limsup_{\ep\,\to\,0^+}\,\frac{\overline{h}^K_\mu(\ep,\ep)}{\log\,(1/\ep)} 
\,\,\,= \,\,\, \limsup_{\ep\,\to\,0^+}\,\frac{\overline{h}^S_\mu(\ep,\ep)}{\log\,(1/\ep)} 
%\\ 
%\underline{\mathrm{mdim}}_Q(X,d,T,LP,\mu) &=&
%\liminf_{\ep\,\to\,0^+}\,\frac{\overline{Q}_\mu(T,d,\ep)}{\log\,(1/\ep)}\,\,\,= \,\,\, 
%\liminf_{\ep\,\to\,0^+}\,\frac{\overline{h}^K_\mu(\ep,\ep)}{\log\,(1/\ep)}\,\,\,
%\,\,\,= \,\,\, \liminf_{\ep\,\to\,0^+}\,\frac{\overline{h}^S_\mu(\ep,\ep)}{\log\,(1/\ep)}.
\end{eqnarray*} 
Consequently, if 
$F\in\big\{\,\overline{h}^K_\mu(\ep),\,\,\overline{h}^S_\mu(\ep) \big\}$, then $$\mathrm{\overline{mdim}_M}(X,d,T)\,\,\,\, 
 =\,\,\,\, \max_{\mu\,\in\,\cP_T(X)}\, \,\,\limsup_{\ep\,\to\,0^+}\,\frac{F(\mu,\ep)}{\log\,(1/\ep)}.$$

 \noindent The corresponding statements for the lower metric mean dimension also hold.

\smallskip

\item[$(b)$] There exists a compact metric space $(X_0,d_0)$ and a homeomorphism $T_0 \colon X_0 \to X_0$ such that 
$$\mathrm{mdim}_M(X_0,d_0,T_0)\,\,\,\, 
 >\,\,\,\, 
 \color{black} \sup_{\mu\,\in\,\cP_{T_0}(X_0)}\, \,\,\limsup_{\ep\,\to\,0^+}\,\frac{F(\mu,\ep)}{\log\,(1/\ep)}$$
 for every 
 \begin{equation*}
    F\,\,\in\,\, \Bigg\{
\begin{aligned}
\,\,\,\underline{h}_\mu^K(\ep,&\delta)\,,\,\,\,
\overline{h}_\mu^K(\ep,\delta)\,,\,\,\,
\underline{h}_\mu^S(\ep,\delta)\,,\,\,\,
\overline{h}_\mu^S(\ep,\delta)\,,\,\,\,
\\&
\overline{h}_\mu^{BK}(\ep)\,,\,\,\,\,\, 
%\cP_\mu(T,d,\ep),\,\,
%PS_\mu(T,\ep),\,\,\,
%\cR_\mu(\ep),\,\,\,
\inf_{\mathrm{diam}(P)\,\le\, \varepsilon} h_\mu(P)\,\,\,
\end{aligned}
\Bigg\}.
\end{equation*} 
\end{itemize}
\end{maintheorem}
\smallskip

%%%%%%%%%%%%%%%%%%%%%%%%%%%%%%%%%%%%%%%%%%%%%%%%%%%%%%%%%%%%%%%%%

\subsection{Ergodic equilibrium states}
The previous discussion suggests another question, previously addressed in \cite{EYZ2025}.

\begin{question}
\emph{Under what conditions on the space $(X,d)$, on the map $T\colon X \to X$, or else on the measure-theoretic $\vep$-entropy map $F$ can we guarantee that 
$$\mathrm{\overline{mdim}_M}(X,d,T) \,\, =\, \sup_{\mu\,\in\,\cE_T(X)}\,\,  \limsup_{\vep \, \to \, 0^+} \, \,\frac{F(\mu, \vep)}{\log \,(1/\vep)} \,\,?$$}
\end{question}

%Admittedly, we are not aware of the existence of any map $F$ that complies with such a variational principle for any dynamics $T\colon X \to X$. 
%The next example describes a setting where this principle is valid generically.

The next two examples describe settings where this principle is valid.

\begin{example}\label{ex:LR}

\emph{Let $(X,d)$ be a compact connected smooth boundaryless manifold with dimension $\mathrm{dim} X \geq 2$, endowed with a Riemannian metric $d$. A Borel probability measure $\lambda$ on $X$ is said to be \emph{good} if $\lambda(\{x\}) = 0$ for every $x \in X$ and $\lambda(U) > 0$ for every non-empty open subset $U$ of $X$. Denote by $\mathrm{Homeo}_\lambda(X)$ the space of homeomorphisms $T \colon X \to X$ that preserve $\lambda$. This is a complete metric space if one considers in $\mathrm{Homeo}_\lambda(X)$ the topology induced by the metric
$$\rho(T_1, T_2) \, = \, \max \big\{d\big(T_1(x), T_2(x)\big), \,d\big(T_1^{-1}(x), T_2^{-1}(x)\big)\big\}.$$}

\emph{G. Lacerda and S. Romaña proved in \cite[Theorems~A~\&~B]{LR24} that a $C^0$-generic homeomorphism $T \in \mathrm{Homeo}_\lambda(X)$ satisfies
\begin{equation}\label{eq:LR1}
    \mathrm{\overline{mdim}_M}(X,d,T) \,\, =\,\,\dim X\,\, =\,\, \limsup_{\vep \, \to \, 0^+} \, \,\frac{ \overline{h}_\lambda^K( \vep)}{\log \,(1/\vep)}.
\end{equation}
Consider the residual subset $\mathfrak{R}$ of $\mathrm{Homeo}_\lambda(X)$ such that, if $T \in \mathfrak{R}$, then $\lambda$ is ergodic (cf.~\cite{OU}) and \eqref{eq:LR1} holds. Combining this information with item $(a)$ of Theorem~\ref{teo:QD}, we conclude that, for every $T \in \mathfrak{R}$, $F\in\cD$ and $D \in \{W_p\colon p \in [1,+\infty)\}\cup \{LP\}$, 
$$\mathrm{\overline{mdim}_M}(X,d,T) \,\, = \,\,
 \max_{\mu\,\in\,\cE_T(X)} \,\mathrm{\overline{mdim}_Q}(X,d,T,D,\mu) \,\, = \,\,
 \max_{\mu\,\in\,\cE_T(X)} \,  \limsup_{\vep\,\to\,0^+}\,\, \frac{F(\mu,\ep)}{\log\,(1/\vep)}$$
 with $\lambda$ as a maximizing measure.}
\end{example}
\smallskip

\begin{example}\label{ex:shift}
\emph{Consider the set $A=[0,1]$ with the Euclidean metric $|\cdot|$, the space $X = A^\ZZ$ endowed with the metric
$$d(x,y) \, = \, \sum_{n \, \in \, \ZZ}\, 2^{-|n|} \, |x_n - y_n|$$
and the standard shift $\sigma \colon X \to X$. Then $\mathrm{\overline{mdim}_M}(X,d,\sigma) = 1$ (cf. \cite{LT18}). Let $\lambda = \mathrm{Leb}^\ZZ$ be the product measure on $[0,1]^\ZZ$ of the Lebesgue measure on $[0,1]$. Then $\lambda$ is $\sigma$-invariant and ergodic; and, due to \cite[E. Example]{LT18} and Theorem~\ref{teo:QD}, for $D \in \{W_p\colon p \in [1,+\infty)\}\cup \{LP\}$ and every $F \in \cD$, 
$$\mathrm{\overline{mdim}_M}(X,d,\sigma) \,\, = \,\,
 \limsup_{\vep\,\to\,0^+}\,\, \frac{F(\lambda,\ep)}{\log\,(1/\vep)}\,\,=\,\,\mathrm{\overline{mdim}_Q}(X,d,\sigma,D,\la) .$$
Therefore, 
%for every $F\in\cD$ and $D \in \{W_p\colon p \in [1,+\infty)\}\cup \{LP\}$, 
$$\mathrm{\overline{mdim}_M}(X,d,\sigma)  
\,\, = \,\,
 \max_{\mu\,\in\,\cE_\sigma(X)} \,  \limsup_{\vep\,\to\,0^+}\,\, \frac{F(\mu,\ep)}{\log\,(1/\vep)}
 \,\, = \,\,  \max_{\mu\,\in\,\cE_\sigma(X)} \,\mathrm{\overline{mdim}_Q}(X,d,\sigma,D,\mu).$$}

\emph{The previous example can be generalized. A compact metric space $(A,\rho)$ is Ahlfors regular if there exists a Borel probability measure $\nu \in \cP(A)$ satisfying
$$\exists\, \vep_0, \,C_1, \,C_2 > 0 \,\colon \quad C_1\, \vep^{\mathrm{dim_B(A,\,\rho)}} \,\,\leq \,\, \nu(B(a,\vep)) \,\,\leq\,\, C_2\, \vep^{\mathrm{dim_B(A,\,\rho)}} \quad \quad \forall a \in A, \,\,\forall 0 < \vep < \vep_0$$
where $B(a,\vep)$ stands for the open ball in $A$ centered at $a$ with radius $\vep$ with respect to the metric $\rho$. For instance, a finite set with the normalized counting measure or the Euclidean cube with the Lebesgue measure are Ahlfors regular. Consider a Ahlfors regular space $(A, \rho, \nu)$, the shift map on the space 
$X = A^\ZZ$ endowed with the metric $d$ and the product measure $\lambda=\nu^\ZZ \in \cE_\sigma(X)$. Given 
$F \in \cD$, we already know that
$$\mathrm{\overline{mdim}_M}(X,d,\sigma) \,\, \geq \,\,
 \sup_{\mu\,\in\,\cE_\sigma(X)} \,  \limsup_{\vep\,\to\,0^+}\,\, \frac{F(\mu,\ep)}{\log\,(1/\vep)}.$$
On the other hand, for any $x \in X$, $\vep \in (0, \vep_0)$ and $n \in \NN$, 
$$B_n(x,\vep) \,\subset \, \big\{y \in X \colon\,\, y_j \in B(x_j,\ep) \quad \forall \, 0 \leq j < n\big\}$$
so,
$$\lambda\big(B_n(x,\vep)\big) \,\,\leq \,\, \prod_{0 \,\leq \,j \,< n}\,\,\nu\big(B(x_j,\ep)\big) \,\,\leq \,\, C_2^n\, \cdot\,\vep^{n\,\mathrm{dim_B(A,\,\rho)}}.$$
These estimates yield (see Definition~\ref{def:Brin-Katok})
$$\overline{h}_\lambda^{BK}(\vep) \,\geq\, -\log C_2 + \mathrm{dim_B(A,\,\rho)} \, \cdot \, \log \, (1/\vep).$$
Therefore,
$$\limsup_{\vep\,\to\,0^+}\,\, \frac{\overline{h}_\lambda^{BK}(\vep)}{\log\,(1/\vep)} \, \, \geq \,\, \mathrm{dim_B(A,\,\rho)} \,\,=\,\,\mathrm{\overline{mdim}_M}(X,d,\sigma).$$
Thus, }

$$
\begin{alignedat}{2}
    \mathrm{\overline{mdim}_M}(X,d,\sigma) \,\,\,&=\,\,\,\,\,\limsup_{\vep\,\to\,0^+}\,\, \frac{F(\lambda,\ep)}{\log\,(1/\vep)}\,\,\,\,\,\, \,\,\,=\,\,\,\,  &\max_{\mu\,\in\,\cE_\sigma(X)} \,  \limsup_{\vep\,\to\,0^+}\,\, \frac{F(\mu,\ep)}{\log\,(1/\vep)}
    \\
    &=\,\,\mathrm{\overline{mdim}_Q}(X,d,\sigma,D,\la) \,\,\,=\,\,  &\max_{\mu\,\in\,\cE_\sigma(X)} \,  \mathrm{\overline{mdim}_Q}(X,d,\sigma,D,\mu).
\end{alignedat}
$$

\end{example}

\begin{remark}
Regarding the lower Brin-Katok entropy $\underline{h}_\lambda^{BK}(\vep)$, R. Yang proved in 
%Será Theorem~3.11 na versão publicada
\cite[Theorem~3.9]{Yang} that, if the Hausdorff and the upper box dimension of the alphabet $(A,\ro)$ coincide, then $$\mathrm{\overline{mdim}_M}(A^\ZZ,d,\sigma) \,\, = \,\,
  \, \limsup_{\vep\,\to\,0^+}\,\sup_{\mu\,\in\,\cE_\sigma(A^\ZZ)}\,\, \frac{\underline{h}_\mu^{BK}(\vep)}{\log\,(1/\vep)}\,\, = \,\,
  \,  \limsup_{\vep\,\to\,0^+}\,\sup_{\mu\,\in\,\cP_\sigma(A^\ZZ)}\,\, \frac{\underline{h}_\mu^{BK}(\vep)}{\log\,(1/\vep)}.$$ For instance, if $(A,\ro)$ is Ahlfors regular then this condition is satisfied. Furthermore, the previous reasoning implies that
  $$\mathrm{\overline{mdim}_M}(A^\ZZ,d,\sigma)\,\, = \,\,
  \limsup_{\vep\,\to\,0^+}\,\, \frac{\underline{h}_\la^{BK}(\vep)}{\log\,(1/\vep)} \,\, = \,\,
  \max_{\mu\,\in\,\cE_\sigma(A^\ZZ)}\,  \limsup_{\vep\,\to\,0^+}\,\, \frac{\underline{h}_\mu^{BK}(\vep)}{\log\,(1/\vep)}.$$
\end{remark}

\subsection{Organization of the paper} In Sections~\ref{se:definition} and \ref{se:mte} we define the main concepts we use and gather some of their properties. After proving Theorems~\ref{teo:QD} and \ref{teo:VPg} in the ensuing sections, we pave the way to proving Theorem~\ref{teo:EYZ} by thoroughly discussing an example in Section~\ref{se:proofD}.

%%%%%%%%%%%%%%%%%%%%%%%%%%%%%%%%%%%%%%%%%%%%%%%%%%%%%%%%%%%%%%%%%%
\section{Basic definitions}\label{se:definition}
Let $(X,d)$ be a compact metric space and $T \colon X \to X$ be a continuous map.

\subsection{Bowen dynamical metrics}\label{sse:nball}

Given $n\in \mathbb{N}$, the dynamical metric $d_n \colon X \times X \, \to \,[0,+\infty)$, determined by the distance $d$ in $X$ and the map $T$, is defined by 
\begin{equation}\label{def:Bdn}
d_n(x,z)=\max\,\Big\{d(x,z),\,d(T(x),T(z)),\,\dots,\,d(T^{n-1}(x),T^{n-1}(z))\Big\}.
\end{equation}
It is easy to check that $d_n$ is indeed a metric and that it generates the same topology as $d$. 

Given $\varepsilon > 0$, $n \in \mathbb{N}$ and a point $x \in X$, the $(n, \varepsilon)$-ball centered at $x$ with radius $\vep$ with respect to the metric $d_n$ (also called Bowen's ball) is the set
\begin{equation}\label{eq:ball}
\mathit{B}_{n}(x, \vep) = \{y \in X \colon \,d_n( x, y) \, < \, \vep\}.
\end{equation}

%%%%%%%%%%%%%%%%%%%%%%%%%%%%%%%%%%%%%%%%%%%%%%%%%%%%%%%%%%%%%%%%% 
\subsection{Separated sets, spanning sets, open covers, closed covers}\label{sse:sepspacov}

Given $\varepsilon > 0$ and $n \in \mathbb{N}$, the next three quantities count the number of orbit segments of length $n$ that are distinguishable at scale $\vep$.

One says that a set $E \subset X$ is $(n,\varepsilon)$-separated if two distinct points in $E$ are at least $\vep$ apart in the metric $d_n$, that is, $d_n(x,z) > \varepsilon$ for every $x\neq z \in E$. We denote by $S(X, n,\varepsilon)$ the maximal cardinality of the $(n,\varepsilon)$-separated subsets of $X$. When $n=1$, we simplify this notation by saying that the set $E$ is $\varepsilon$-separated, and denoting the maximal cardinality of the $\vep$-separated subsets of $X$ by $S(X,d,\vep)$. We note that, for every $\vep>0$ and $n\in\mathbb N$,
$$S(X, n, \vep) \,=\, S(X, d_n,\vep).$$

A set $Y\subset X$ is said to be $(n,\varepsilon)$-spanning if for any $x\in X$ there exists $z\in Y$ such that $d_n(x,z)<\varepsilon$. We denote by $R(X,n,\varepsilon)$ the minimal cardinality of the $(n,\varepsilon)$-spanning subsets of $X$. Again, when $n=1$, we refer to $\varepsilon$-spanning sets, with minimal cardinality denoted $R(X, d,\vep)$, and one has
$$R(X, n, \vep) \,=\, R(X, d_n,\vep) \quad \forall \,\vep>0 \,\,\,\forall\, n\in\mathbb{N}.$$

Given $\vep >0$ and $n \in \mathbb{N}$, let $\mathcal{N}(X,n,\vep)$ be the minimum cardinality of any cover of $X$ by open balls with radius $\vep$ in the metric $d_n$. When $n=1$, we refer to $\mathcal{N}(X,d,\vep)$, so
$$\mathcal{N}(X, n, \vep) \,=\, \mathcal{N}(X, d_n,\vep) \quad \forall \,\vep>0 \,\,\,\forall\, n\in\mathbb{N}.$$
In what follows, if we need to use closed balls instead of open ones, we change the notation to $\mathcal{C}(X,d,\vep)$. We note that, given $\vep > 0$,
\begin{equation}\label{eq:NC}
\mathcal{N}(X,d,2\vep)\,\leq\, \mathcal{C}(X,d,\vep)\, \leq \, \mathcal{N}(X,d,\vep). 
\end{equation} 

Due to the compactness of $X$, the values of $S(X,n,\varepsilon)$, $R(X,n,\varepsilon)$, $\mathcal{N}(X,n,\vep)$ and $\mathcal{C}(X,n,\vep)$ are finite for every $n\in\mathbb{N}$ and $\vep>0$. In addition, the previous three quantities are related by the following inequalities (cf. \cite[\S 7]{Wa}):
\begin{equation}\label{eq:srn} 
\mathcal{N}(X,n,2\vep) \, \leq \, R(X,n,\vep) \, \leq \, S(X,n,\vep)\, \leq \, R(X,n,\vep/2) \, \leq \, \mathcal{N}(X,n,\vep/2).    
\end{equation}
Moreover, 
\begin{equation}\label{eq:cov} 
\mathcal{N}(X,n+m,\vep) \, \leq \, \mathcal{N}(X,n,\vep)\, . \, \mathcal{N}(X,m,\vep)
\end{equation}
so, the sequence $\big(\log \mathcal{N}(X,n,\vep)\big)_{n\, \in\, \mathbb{N}}$ is sub-additive.

%%%%%%%%%%%%%%%%%%%%%%%%%%%%%%%%%%%%%%%%%%%%%%%%%%%%%%%%%%%%%%%%%%
\subsection{Box-counting dimension}\label{sse:boxdim}

Given $\vep > 0$, recall that $\mathcal{N}(X,d,\vep)$ stands for the smallest number of open balls in the metric $d$ with radius $\vep$ that cover $X$. We observe that, for every $\vep >0$, one has
\begin{equation}\label{eq:box}
S(X,d,2\vep) \, \leq \, \mathcal{N}(X,d,\vep) \, \leq\, S(X,d, \vep/2).
\end{equation}

The upper box-counting dimension of $(X,d)$ is defined by
$$\mathrm{\overline{dim}_B}(X,d) \,=\, \limsup_{\vep \, \to \, 0^+}\,\frac{\log \mathcal{N}(X,d,\vep)}{\log\,(1/\varepsilon)}.$$
The lower box-counting dimension $\mathrm{\underline{dim}_B}(X,d)$ is defined by taking $\liminf$ instead of $\limsup$. Notice that, due to \eqref{eq:box}, the upper box-counting dimension may be computed using $\vep$-separated sets, that is,
$$\mathrm{\overline{dim}_B}(X,d) \,=\, \limsup_{\vep \, \to \, 0^+}\,\frac{\log S(X,d,\vep)}{\log\,(1/\varepsilon)}$$
(and similarly for the lower version).

%%%%%%%%%%%%%%%%%%%%%%%%%%%%%%%%%%%%%%%%%%%%%%%%%%%%%%%%%%%%%%%%%%
\subsection{Metric mean dimension}\label{sse:mdim}

The \emph{upper and lower metric mean dimensions} are labels for dynamical systems introduced by E. Lindenstrauss and B. Weiss in \cite{LW2000} to quantify the complexity of infinite entropy systems. We denote them by $\mathrm{\overline{mdim}_M}(X,d,T)$ and $\mathrm{\underline{mdim}_M}(X,d,T)$, respectively, to emphasize their dependence on the fixed metric $d$ of the space $X$ where the dynamics $T$ acts. They are defined by
\begin{eqnarray*}
\mathrm{\mathrm{\overline{mdim}_M}}\big(X,d,T\big) &=& \limsup_{\vep \, \to \, 0^+}\,\frac{\limsup_{n\, \to \, +\infty}\, \frac{1}{n}\,\log S(X,n,\vep)}{\log\,(1/\varepsilon)}\\
\mathrm{\underline{mdim}_M}\big(X,d,T\big) &=& \liminf_{\vep \, \to \, 0^+}\,\frac{\limsup_{n\, \to \, +\infty}\,\frac{1}{n}\,\log S(X,n,\vep)}{\log\,(1/\varepsilon)}.
\end{eqnarray*}
Due to \eqref{eq:srn}, the metric mean dimension may be computed using $R(X,n,\vep)$ or $\mathcal{N}(X,n,\vep)$ instead of $S(X,n,\vep)$. 
If  $\mathrm{\mathrm{\overline{mdim}_M}}\big(X,d,T\big)$ and $\mathrm{\mathrm{\underline{mdim}_M}}\big(X,d,T\big)$ agree, then the common value is denoted by $\mathrm{\mathrm{mdim}_M}\big(X,d,T\big)$ and called metric mean dimension.

\smallskip

In what follows, we simplify the notation by denoting
\begin{equation}\label{eq:he}
h_\vep(X,d,T) \, = \, \limsup_{n\, \to \, +\infty}\, \frac{1}{n}\,\log\, \mathcal{Z}(X,n,\vep)
\end{equation}
where $\mathcal{Z}$ is chosen from the set $\{\mathcal{N}, R, S\}$ according to the context.

%%%%%%%%%%%%%%%%%%%%%%%%%%%%%%%%%%%%%%%%%%%%%%%%%%%%%%%%%%%%%%%%%%
\subsection{$W_p$ and $LP$ metrics on $\mathcal{P}(X)$}\label{sse:LP}

It is known that the space $\mathcal{P}(X)$ of the Borel probability measures on $X$ is nonempty; and it is compact if endowed with the weak$^*$-topology. Moreover, there are metrics on $\mathcal{P}(X)$ inducing this topology, the classic ones being the \emph{Wasserstein distances} and the \emph{L\'evy-Prokhorov distance}. The former are defined by
$$\mathrm{W_p}(\mu,\nu) = \inf_{\pi\,\in\,\Pi(\mu,\nu)}\,\left(\int_{X\times X}\,\big(d(x,y)\big)^p\;d\pi(x,y)\right)^{1/p}$$
where $p \in [1, +\infty)$ and $\Pi(\mu,\nu)$ denotes the set of probability measures on the product space $X\times X$ with marginals $\mu$ and $\nu$ (see \cite{Vi} and references therein for more details). The latter is defined by
$$LP(\mu,\nu) = \inf\,\Big\{\varepsilon > 0 \colon \,\, \forall\, E\subset X \quad  \forall\, \,\text{$\varepsilon$-neighborhood $V_\varepsilon(E)$ of $E$ one has} $$
$$\quad \quad \quad \quad \quad \quad \quad \quad \quad \nu(E)\leq \mu(V_\varepsilon(E))+\varepsilon \quad \text{ and } \quad \mu(E)\leq \nu(V_\varepsilon(E))+\varepsilon \Big\}.$$

\noindent These distances are compared by (cf. \cite[p.~11]{BB})
\begin{equation}\label{MetricRelations}
\begin{aligned}
W_q \,\,\,&\le \,\,\,W_p \,\,\,
\le \,\,\,\diam(X,d)^{1-\frac{q}{p}}\, W_q^{\frac{q}{p}}
\qquad & \forall\, 1 \le q \le p \\
LP^{1+\frac{1}{p}} \,\,&\le\,\,\, W_p\,\,\,
\le \,\,\bigl(1+\diam(X,d)^p\bigr)^{\frac{1}{p}}\, LP^{\frac{1}{p}} \qquad & \forall\, p \in [1, +\infty).
\end{aligned}
\end{equation}

\noindent We refer the reader to \cite{ GSu, Vi} where more information can be found on these metrics.

%%%%%%%%%%%%%%%%%%%%%%%%%%%%%%%%%%%%%%%%%%%%%%%%%%%%%%%%

\section{Measure-theoretic $\vep$-entropies}\label{se:mte}

Let $(X,d)$ be a compact metric space, $T \colon X \to X$ a homeomorphism and $\mu \in \cP_T(X)$. Recall from Subsection~\ref{sse:nball} the definitions of dynamical metric $d_n$ and $(n, \varepsilon)$-ball $B_n(x,\vep)$.

\subsection{$(\ep,\delta)-$Katok entropy}\label{se:KE}

For $0 < \delta < 1$, $n \in \NN$ and $\vep > 0$, we denote by $\mathcal{N}_\mu(n,\vep, \delta)$ the minimum number of open $(n,\vep)$-balls with radius $\vep$ in the dynamical metric $d_n$  needed to cover any set of $\mu$-measure strictly bigger than $1 - \delta$.

\begin{definition}\label{def:Kerg}
\emph{The \emph{upper} and \emph{lower $(\ep,\delta)-$Katok entropy of $\mu$} are given, respectively, by
\begin{equation}\label{eq:KatokEntropy}
\overline{h}_\mu^K(\vep, \delta) \, = \, \limsup_{n \, \to \, +\infty}\, \frac{\log \, \mathcal{N}_\mu(n,\vep, \delta)}{n} \qquad \text{ and } \qquad \underline{h}_\mu^K(\vep, \delta) \, = \, \liminf_{n \, \to \, +\infty}\, \frac{\log \, \mathcal{N}_\mu(n,\vep, \delta)}{n}.
\end{equation} 
The \emph{upper} and \emph{lower Katok entropy of $\mu$ at scale $\ep$} are defined by 
\begin{equation}
    \begin{alignedat}{3}
        \overline{h}_\mu^K(\vep)\,\,&=\,\,  \lim_{\delta\,\to\,0^+}\overline{h}_\mu^K(\vep,\delta)\,\,&=\,\,  \sup_{\delta\,\to\,0^+}\overline{h}_\mu^K(\vep,\delta)        \\
        \underline{h}_\mu^K(\vep)\,\,&=\,\,  \lim_{\delta\,\to\,0^+}\underline{h}_\mu^K(\vep,\delta)\,\,&=\,\,  \sup_{\delta\,\to\,0^+}\underline{h}_\mu^K(\vep,\delta)\mathrlap{.}
    \end{alignedat}
\end{equation}}
\end{definition}

\begin{remark}
    In the literature, we can find an alternative definition of Katok entropy, whose values for ergodic measures are as above but the extension to invariant probabilities is done by integrating over their ergodic decompositions. In particular, the resulting entropy-like map is convex. Thus, the corresponding measure-theoretic metric mean dimension does not satisfy a classical variational principle (cf. Proposition~\ref{prop:conv}).
\end{remark}

%%%%%%%%%%%%%%%%%%%%%%%%%%%%%%%%%%%%%%%%%%%%%%%%%%%%%%%%

\subsection{$(\ep,\delta)$-Shapira entropy}

Given a finite open cover $\cU$ of $X$, the diameter of $\cU$, which we denote by $\diam(\cU,d)$, is given by  $\max_{U\,\in\,\cU}\,\diam(U,d)$, where the diameter $\diam(A,d)$ of a set $A \subset X$ is the supremum of the distances between pairs of points in $A$. In what follows, $\mathrm{Leb}(\cU)$ stands for the Lebesgue number of $\cU$, that is, the largest $r>0$ such that every open ball of radius $r$ is contained in some element of $\cU$.

\smallskip

Take $\mu \in \cP_T(X)$. For $0 < \delta < 1$ and a finite open cover $\cU$ of $X$, we denote by $\mathcal{N}_\mu(\cU, \delta)$ the minimum number of elements of $\cU$ whose union has a $\mu$-measure strictly greater than $1 - \delta$. Set $$\overline{h}^S_\mu(\cU,\delta)\,=\, \limsup_{n \, \to \, +\infty}\, \frac{\log \, \mathcal{N}_\mu(\cU^n, \delta)}{n}\qquad\text{and}\qquad \underline{h}^S_\mu(\cU,\delta)\,=\, \liminf_{n \, \to \, +\infty}\, \frac{\log \, \mathcal{N}_\mu(\cU^n, \delta)}{n}$$ 
where $\cU^n\,=\,\bigvee_{j=0}^{n-1}\,T^{-j}\big(\cU\big)$. 
\smallskip

U. Shapira proved in \cite{Shapira} that, if $\mu$ is ergodic, then the above limits exist and are independent of $\delta$. However, since we are interested in all invariant measures and their interplay with $\delta$, we keep it in the notation. 

\begin{definition}\label{def:Shapira}
\emph{The \emph{upper} and \emph{lower $(\ep,\delta)$-Shapira entropies of $\mu$} are given by}
\begin{equation}\label{eq:ShapiraEntropy}
\overline{h}_\mu^S(\vep, \delta) \, = \, \inf_{\diam(\cU,d)\,\leq\,\ep}\, \overline{h}^S_\mu(\cU,\delta)\qquad\text\qquad \underline{h}_\mu^S(\vep, \delta) \, = \, \inf_{\diam(\cU,d)\,\leq\,\ep}\, \underline{h}^S_\mu(\cU,\delta).
\end{equation}
\end{definition}

The next lemma describes the connection between the Katok and Shapira entropies.

\begin{lemma} \label{K=S}
For every $\mu\in\cP_T(X)$, $\ep>0$ and $0 < \delta < 1$, 
    $$\overline{h}_\mu^K(\vep, \delta)\,\,\leq\,\,\overline{h}_\mu^S(\vep, \delta)\,\,\leq\,\,\overline{h}_\mu^K(\vep/4, \delta)\,\,\leq\,\,\overline{h}_\mu^K(\ep/4,\delta/4).$$ and $$\underline{h}_\mu^K(\vep, \delta)\,\,\leq\,\,\underline{h}_\mu^S(\vep, \delta)\,\,\leq\,\,\underline{h}_\mu^K(\vep/4, \delta)\,\,\leq\,\,\underline{h}_\mu^K(\ep/4,\delta/4).$$
\end{lemma}

\begin{proof}
We prove the inequalities in the first line; a similar reasoning shows the others. Given $\mu\in\cP_T(X)$, $\ep>0$ and $0 < \delta < 1$, fix $n\in\NN$ and $\cU$ such that $\diam(\cU,d)\leq\ep$. For each nonempty $U\in\cU^n$, take $x_U\in U$. Then $U\subset B_n(x_U,\ep)$ and $$\cN_\mu(n,\ep,\delta)\,\,\leq\,\,\min\big\{\cF\subset\cU^n\,\colon\,\mu\big(\cup_{U\in\cF }B_n(x_U,\ep)\big)>1-\delta\big\}\,\,\leq\,\,\cN_\mu(\cU^n,\delta).$$
Thus, 
$$\overline{h}_\mu^K(\vep, \delta)\,\leq\,\overline{h}_\mu^S(\cU, \delta)$$ 
and the first inequality follows by taking the infimum over $\cU.$ 
\smallskip

Regarding the remaining inequalities, let $\cU$ be a cover of $X$ satisfying $\diam(\cU,d)\leq\ep$ and $\mathrm{Leb}(\cU)\geq\ep/4$ (which always exists, cf. \cite[Lemma 3.4]{GS}). Since every Bowen's ball $B_n(x,\ep/4)$ is contained in some element of $\cU^n$, we have $\cN_\mu(\cU^n,\delta)\leq \cN_\mu(n,\ep/4,\delta)$. Thus, $$\overline{h}_\mu^S(\ep, \delta)\,\,\leq\,\,\overline{h}_\mu^S(\cU, \delta)\,\,\leq\,\,\overline{h}_\mu^K(\vep/4, \delta)\,\,\leq\,\,\overline{h}_\mu^K(\ep/4,\delta/4).$$
\end{proof}

%%%%%%%%%%%%%%%%%%%%%%%%%%%%%%%%%%%%%%%%%%%%%%%%%%%%%%%%%%%%%%%%%%
\subsection{Brin-Katok $\vep$-entropy}

Given $\mu \in \cP_T(X)$, this notion of complexity is based on the average decay rate of the size, measured by $\mu$, of a dynamical ball in $X$. 
    
\begin{definition}\label{def:Brin-Katok}
\emph{ For $x\in X$ and $\ep>0$, the \emph{upper} and \emph{lower Brin-Katok entropies of $\mu \in \cP_T(X)$ at $x$ with respect to the scale $\ep$} are given by $$\overline{h}^{BK}_\mu(\ep,x)\,=\, \limsup_{n \, \to \, +\infty}\, -\,\frac{\log \, \mu\big(B_n(x,\ep)\big)}{n}\qquad\text{and}\qquad \underline{h}^{BK}_\mu(\ep,x)\,=\, \liminf_{n \, \to \, +\infty}\, -\,\frac{\log \, \mu\big(B_n(x,\ep)\big)}{n}$$
and the \emph{upper} and \emph{lower Brin-Katok entropies of $\mu$ at scale $\ep$} are given, respectively, by $$\overline{h}^{BK}_\mu(\ep)\,=\, \int\overline{h}^{BK}_\mu(\ep,x)\,d\mu(x)\qquad\text{and}\qquad \underline{h}^{BK}_\mu(\ep,x)\,=\, \int\underline{h}^{BK}_\mu(\ep,x)\,d\mu(x).$$}
\end{definition}

\smallskip

The following two properties will be useful later on.

\smallskip

\begin{lemma}\label{BK_leq_ep-entropy}
Let $\mu\in\cP_T(X)$ and $\ep>0$. Then $$\overline{h}^{BK}_\mu(\ep,x)\,\,\leq\,\, \limsup_{n\, \to \, +\infty}\, \frac{1}{n}\,\log S(X,n,\vep/2)\qquad\text{at }\,\mu\text{-a.e. }x\,\in\,X.$$
In particular, $$\overline{h}^{BK}_\mu(\ep)\,\,\leq\,\, \limsup_{n\, \to \, +\infty}\, \frac{1}{n}\,\log S(X,n,\vep/2).$$
\end{lemma}

\begin{proof}
To simplify notation, set 
$$h_{\ep/2} \,=\,\limsup_{n\, \to \, +\infty}\, \frac{1}{n}\,\log S(X,n,\vep/2).$$
For each $\ga>0$, let $$A_\ga\,=\,\{x\,\in\,X\,\colon\, \overline{h}^{BK}_\mu(\ep,x) \,>\,h_{\ep/2}\,+\,\ga\}.$$
\medskip

\begin{claim} \emph{$\quad\mu(A_\ga)\,=\,0$ for every $\ga>0.$}
\end{claim}

\begin{proof}[Proof of the Claim:]
Fix $\ga>0$. For each $n\in\NN,$ let 
\begin{eqnarray*}
A_\ga^{(n)}&=&\,\big\{x\,\in\,X\,\colon\, -\frac{1}{n}\,\log\mu\big(B_n(x,\ep)\big) \,>\,h_{\ep/2}\,+\,\ga\big\} \\ &=&\,\big\{x\,\in\,X\,\colon\, \mu\big(B_n(x,\ep)\big) \,<\,e^{-n(h_{\ep/2}\,+\,\ga)}\big\}.
\end{eqnarray*}
Take $n_\ga\in\NN$ such that $$S(X,n,\vep/2)\,\,\leq\,\,e^{n(\,h_{\ep/2}\,\,+\,\frac{\ga}{2}\,)}\qquad\forall\, n\,\geq\, n_\ga.$$
Hence, 
$$A_\ga=\bigcap_{N\,\geq\,n_\ga} \,\bigcup_{n\,\geq\,N} \,A_\ga^{(n)}.$$ 
Fix $N\,\geq\,n_\ga$. For each $n\,\geq\,N$, let $E_n$ be a maximal $(n,\ep/2)$-separated subset of $X.$ Then $$A_\ga^{(n)}\,\subset\,\bigcup_{y\,\in\,E_n}\,B_n(y,\ep/2)\qquad\quad\text{and}\quad\qquad \#E_n\,\,\leq\,\,e^{n\,(\,h_{\ep/2}\,\,+\,\frac{\ga}{2}\,)}.$$ For each $x\in A_\ga^{(n)}$, take $
y\,\in\,E_n$ satisfying $x\,\in\,B_n(y,\ep/2)$. Then $B_n(y,\ep/2)\,\subset\,B_n(x,\ep)$, which implies that $\mu\big(B_n(y,\ep/2)\big)\,<\,e^{-n(h_{\ep/2}\,+\,\ga)}$. Define $$E^\prime_n\,\,=\,\,\Big\{y\,\in\,E_n\,\colon\,\mu\big(B_n(y,\ep/2)\big)\,<\,e^{-n(h_{\ep/2}\,+\,\ga)} \Big\}.$$ 
Then $A_\ga^{(n)}\,\subset\,\bigcup_{y\,\in\,E^\prime_n}\,B_n(y,\ep/2)$ and we get 
$$\mu\big( A_\ga^{(n)}\big)\,\,\leq\,\,\sum_{y\,\in\,E^\prime_n}\,\mu\big(B_n(y,\ep/2)\big)\,\,\leq\,\,\#E'_n\,\,\,e^{-n(h_{\ep/2}\,+\,\ga)}\,\,\leq\,\,\#E_n\,\,\,e^{-n(h_{\ep/2}\,+\,\ga)}\,\,\leq\,\,e^{\frac{-n\ga}{2}}.$$ 
Thus, 
$$\mu\big( A_\ga \big)\,\,\leq\,\,\mu\Big( \bigcup_{n\,\geq\,N} \,A_\ga^{(n)} \Big)\,\,\leq\,\,\sum_{n\,\geq\,N} \,e^{\frac{-n\ga}{2}}\,\,=\,\, \frac{e^{\frac{-N\ga}{2}}}{1\,-\,e^{\frac{-\ga}{2}}}\qquad\forall \,N\,\geq\,n_\ga.$$

\noindent We finish the proof of the claim by taking $N\,\to\,+\infty$.
\end{proof}

Now, Lemma~\ref{BK_leq_ep-entropy} is a straightforward consequence of the claim since, given  $\mu\in\cP_T(X)$ and $\ep>0$, 
$$X \setminus \Big\{x \in X \colon \,\,\overline{h}^{BK}_\mu(\ep,x) \,\leq\, \limsup_{n\, \to \, +\infty}\, \frac{1}{n}\,\log S(X,n,\vep/2)\Big\} \,\, \subseteq \bigcup_{\ga \,\in \,\mathbb{Q}\,\cap\,(0,+\infty)}  A_\ga.$$
\end{proof}

\begin{lemma}\label{BK_convex}
For every $\ep>0$, the maps 
$$\mu\,\in\,\cP_T(X)\,\,\,\longmapsto\,\,\,\overline{h}_\mu^{BK}(\ep) \quad \quad \text{and} \quad \quad \mu\,\in\,\cP_T(X)\,\,\,\longmapsto\,\,\,\underline{h}_\mu^{BK}(\ep)$$
are convex. 
\end{lemma}

\begin{proof}
We prove the assertion for the upper Brin-Katok entropy; the reasoning concerning the other map is similar. Fix $\vep > 0$. Let $(\mu_k)_{k\,\in\,\NN}$ be a sequence of invariant probability measures and $(t_k)_{k\,\in\,\NN}$ be a sequence of positive real numbers such that $\sum_{k\,\in\,\NN} \,t_k\,=\,1$. Consider $\mu\,=\,\sum_{k\,\in\,\NN} \,t_k\,\mu_k\,\in\,\cP_T(X).$ For each $k\in \NN$, $\mu\,\geq\,t_k\,\mu_k$. Therefore, for all $x\in X$ and $n\in \NN$, $$\log\,\mu\big(B_n(x,\ep)\big)\,\,\geq\,\,\log\,\mu_k\big(B_n(x,\ep)\big)\,+\,\log\,t_k$$
and so
$$\overline{h}_\mu^{BK}(\ep,x)\,\,\leq\,\,\overline{h}_{\mu_k}^{BK}(\ep,x)\qquad \forall\, x\in X.$$ 
\smallskip

\noindent We finish the proof by integrating both sides of the previous inequality with respect to $\mu_k$, multiplying both sides by $t_k$ and adding up over $k\in\NN.$
\end{proof}

%%%%%%%%%%%%%%%%%%%%%%%%%%%%%%%%%%%%%%%%%

\subsection{Rényi entropy}
Denote by $P$ any finite Borel measurable partition of $X$, by $\diam(P)$ its diameter, defined by $\max_{A \,\in \, P} \, \diam A$ and by $P^n$ the refined partition $\bigvee_{j=0}^{n-1}\, T^{-j}(P)$. Given $\mu \in \cP_T(X)$, the measure-theoretic entropy of $\mu$ with respect to $P$ is defined by
$$h_\mu(P)\,=\,\lim_{n \, \to \, +\infty} \, \frac{1}{n}\,H_\mu(P^n)$$
where $H_\mu(Q)\,=\,-\sum_{\mathrm{Q}\,\in\,Q}\,\mu(\mathrm{Q})\,\log\mu(\mathrm{Q})$ stands for the Shannon entropy.

%The measure-theoretic entropy of $\mu$ is given by $$h_\mu(T) \, = \, \sup_P \,h_\mu(P)$$ where the supremum is taken over all finite Borel measurable partitions $P$ of X. 

\begin{definition}
\emph{Given $\vep>0$, the \emph{Rényi $\vep$-entropy of $\mu$ at scale $\ep$} is defined by
$$\inf_{\diam(P)\, \leq\, \vep} \, h_\mu(P)$$
where the infimum is taken over all finite Borel measurable partitions of X with diameter at most $\vep$.}
\end{definition}
We refer the reader to \cite[Section~3]{GS} for more information on this concept.

%%%%%%%%%%%%%%%%%%%%%%%%%%%%%%%%%%%%%%%%%%%%%%%%%%%%%%%%%%%%%%%%%%
\section{Proof of Theorem~\ref{teo:VPg}}\label{se:proofG}

Assume that $F \colon \cK \times (0, +\infty)\, \to [0, +\infty]$ satisfies the s.m.m.d.~property and let $C>0$ be the constant provided by Definition~\ref{def:MP}. 
\smallskip

Clearly, for every $\mu \in \cK$, one has
$$\limsup_{\vep \, \to \, 0^+} \, \frac{\sup_{\mu \, \in \, \cK}\,F(\mu, \vep)}{\log \, (1/\vep)}\, \geq\, \limsup_{\vep \, \to \, 0^+} \, \frac{F(\mu, \vep)}{\log \, (1/\vep)}$$
and so
\begin{equation}\label{eq:geq}
\limsup_{\vep \, \to \, 0^+} \, \frac{\sup_{\mu \, \in \, \cK}\,F(\mu, \vep)}{\log \, (1/\vep)}\,\, \geq\, \sup_{\mu\,\in\,\cK}\,\limsup_{\ep\,\to\,0^+}\,\frac{F(\mu,\ep)}{\log(1/\ep)}.
\end{equation}
\smallskip

Regarding the converse inequality, set
$$\mathcal{L} \, =\, \limsup_{\vep \, \to \, 0^+} \, \frac{\sup_{\mu \, \in \, \cK}\,F(\mu, \vep)}{\log \, (1/\vep)}.$$
If $\mathcal{L}=0$, then \eqref{eq:geq} already ensures that
$$\limsup_{\vep \, \to \, 0^+} \, \frac{\sup_{\mu \, \in \, \cK}\,F(\mu, \vep)}{\log \, (1/\vep)}\,\, =\, \sup_{\mu\,\in\,\cK}\,\limsup_{\ep\,\to\,0^+}\,\frac{F(\mu,\ep)}{\log(1/\ep)}.$$
If $\mathcal{L}>0$,  let $(\ep'_\ell)_{\ell \, \in \mathbb{N}}$ be a sequence of scales converging to zero such that 
\begin{equation*}
\mathcal{L}\,=\, \lim_{\ell\,\to\,+\infty}\,\frac{\sup_{\mu\,\in\,\cK}\, F(\mu,\ep'_\ell)}{\log\,(1/\ep'_\ell)}.
\end{equation*} 
Taking a subsequence if necessary, we can assume that 
$$\ep'_\ell \, < \, 2^{-\ell} \quad \quad \forall\, \ell \in \NN.$$

\smallskip

Fix a sequence $(t_\ell)_{\ell \, \in \mathbb{N}}$ of positive real numbers with polynomial decay to zero as $\ell$ goes to $+\infty$ and such that $\sum_{\ell \, \in \mathbb{N}} \,t_\ell\,=\,1$. For example, we may choose $t_\ell\,=\,6/(\pi\,\ell)^2$ for every $\ell \in \NN$. Then, since $\ep'_\ell \, < \, 2^{-\ell}$ for every $\ell \in \NN$, the sequence $\ep_\ell\,= t_\ell\,\ep'_\ell/C$ satisfies
$$\lim_{\ell\,\to\,+\infty}\,\frac{\log\,(1/\ep_\ell)}{\log\,(1/\ep'_\ell)} \, = \, 1$$
and so 
$$\mathcal{L} \,\,=\,\, \lim_{\ell\,\to\,+\infty}\,\frac{\sup_{\mu\,\in\,\cK}\, F(\mu,\ep'_\ell)}{\log\,(1/\ep'_\ell)} \,\,=\,\,  \lim_{\ell\,\to\,+\infty}\,\frac{\sup_{\mu\,\in\,\cK}\, F(\mu,C\,\ep_\ell\,/t_\ell)}{\log\,(1/\ep_\ell)}.$$

\medskip

\noindent \textbf{Case 1}: $\mathcal{L} < +\infty$. \\

For every $N \in \NN$, there exists $L(N)\in\NN$ such that 
\begin{equation}\label{L(N)condition}
\sup_{\mu\,\in\,\cK}\, F(\mu,C\,\ep_\ell\,/t_\ell)\,\, > \,\,\Big(1-\frac{1}{N}\Big)\,\cdot\,\mathcal{L}\,\cdot \,\log\,(1/\ep_\ell) \qquad \forall\,\ell \geq L(N).
\end{equation} 
Without loss of generality, we may assume that $$L(N+1)\,>\,L(N) \quad \quad \forall\, N\in\NN.$$
Consider the subsequence $(\ep_{L(N)})_{N\,\in\,\NN}$ of $(\ep_\ell)_{\ell \,\in\, \NN}$ and a sequence of measures $(\mu_N)_{N\,\in\,\NN}$ in $\cK$ such that 
\begin{equation}\label{eq:F}
F\big(\mu_N,C\,\ep_{L(N)}\,/t_{L(N)}\big) \,\,>\,\, 
    \Big(1-\frac{1}{N}\Big)\,\cdot\,
\mathcal{L}\,\cdot\,\log\,(1/\ep_{L(N)}) \qquad\forall N\in\NN.  
\end{equation}
Let $\mu_0\in\cK$ be defined by $$\mu_0\,=\,\beta\,\cdot\sum_{N\,\in\,\NN}\,t_{L(N)}\,\mu_N$$
where $\beta$ is a normalization constant that is greater than or equal to $1$ since each natural number appears at most once in the sequence $\big({L(N)}\big)_{N\,\in\,\NN}$. Then 
$$\mu_0\,\geq \,t_{L(N)}\,\mu_N \quad \quad \forall\, N\in\NN$$
so, by the s.m.m.d.~property in $\cK$ and the inequality \eqref{eq:F}, we get 
$$F\big(\mu_0,\ep_{L(N)}\big) \,\,\geq \,\, F\big(\mu_N,C\,\ep_{L(N)}\,/t_{L(N)}\big)  \,\,>\,\,
    \Big(1-\frac{1}{N}\Big)\,\cdot\,
\mathcal{L}\,\cdot\,\log\,(1/\ep_{L(N)}) \qquad\forall\, N\in\NN.$$
By dividing the previous estimates by $\log\,(1/\ep_{L(N)})$ and taking $N\,\to\,+\infty$, we obtain 
\begin{equation}\label{eq:leq}
\limsup_{\ep\,\to\,0^+}\,\frac{F(\mu_0,\ep)}{\log\,(1/\ep)}\,\,\geq\,\,\limsup_{N\,\to\,+\infty}\,\frac{F(\mu_0,\ep_{L(N)})}{\log\,(1/\ep_{L(N)})}\,\,\geq\,\,\mathcal{L}.
\end{equation}
Bringing together \eqref{eq:geq} and \eqref{eq:leq}, we conclude that
$$\mathcal{L}\,\, =\, \max_{\mu\,\in\,\cK}\,\limsup_{\ep\,\to\,0^+}\,
\frac{F(\mu,\ep)}{\log(1/\ep)}.$$

\smallskip

\noindent \textbf{Case 2}: $\mathcal{L} = +\infty$. \\

The proof in this case proceeds in a similar way after replacing $\big(1-\frac{1}{N}\big)\,\cdot\,\mathcal{L}$ on the right-hand side of \eqref{L(N)condition} by $N$.
\smallskip

Assuming that $F$ is non-increasing in $\vep$, the limits
$$\limsup_{\ep\,\to\,0^+}\,
\frac{F(\mu,\ep)}{\log(1/\ep)}\qquad\text{and}\qquad\limsup_{\ep\,\to\,0^+}\,
\frac{\sup_{\mu\,\in\,\cK}\,F(\mu,\ep)}{\log(1/\ep)}$$ are attained at any sequence $\ep_k\rightarrow 0^+$ satisfying $$\lim_{k\,\to\,+\infty}\frac{\log \ep_{k+1}}{\log \ep_{k}}\,\,=\,\,1.$$
The equality \eqref{eq:vpnon} is deduced by an analogous reasoning. 
\smallskip

The strong convex geometry of $\cK_{max}$ described in (c) is an immediate consequence of \eqref{StrongConcavity}.

%%%%%%%%%%%%%%%%%%%%%%%%%%%%%%%%%%%%%%%%%%%%%%%%%%%%%%%%%%%%%%%%%%
\section{Proof of Theorem~\ref{teo:QD}}\label{se:MQD}

This section is divided into two parts, corresponding to the choice of the Wasserstein or L\'evy-Prokhorov distance to metrize the weak$^*$-topology. 

\subsection{$W_p$-mean quantization dimension}\label{se:QW}

Definition~\ref{def:qnumber} admits the following reformulation when $D$ is a $p$-Wasserstein metric for some $p \in [1, +\infty)$.

\begin{lemma}\cite[Proposition 3.2]{BB} The quantization number $Q_{\mu,W_p}(\ep)$ for the metric $W_p$ is the minimal cardinality of a set $F\subset X$ satisfying 
$$\int   d(x,F)^p\, d\mu(x)\,\,\leq\,\,\ep^p$$
where $d(x,F)\,=\,\min_{y\,\in\,F}\,d(x,y).$
\end{lemma}

The next lemma relates the Katok entropies to the dynamical quantization numbers with respect to the Wasserstein distances.

\begin{lemma}\label{QuantGrowthRate}
Let $(X,d)$ be a compact metric space and $T\colon X \to X$ be a continuous map. For every $\mu\,\in\,\cP_T(X)$ and $p\,\in[1, +\infty)$,
\begin{equation}
\begin{alignedat}{4}\label{ineq}
\overline{h}^K_\mu(2\ep,3/4)\,\,\,
&\leq\,\,\,\overline{Q}_\mu(T,W_1,\ep)\,\,\,
&\leq\,\,\,\overline{Q}_\mu(T,W_p,\ep)\,\,\,
&\leq\,\,\,\overline{h}^K_\mu(2^{-1/p}\ep)
%\,=\,\sup_{\delta\,\in \, (0,1)}\overline{h}^K_\mu(2^{-1/p}\ep, \delta) 
\\ \nonumber
\underline{h}^K_\mu(2\ep,3/4)\,\,\,
&\leq\,\,\,\underline{Q}_\mu(T,W_1,\ep)\,\,\,
&\leq\,\,\,\underline{Q}_\mu(T,W_p,\ep)\,\,\,
&\leq\,\,\,\underline{h}^K_\mu(2^{-1/p}\ep).
%\,=\,\sup_{\delta\,\in \, (0,1)}\underline{h}^K_\mu(2^{-1/p}\ep, \delta)
\end{alignedat}
\end{equation}
\end{lemma}

\begin{proof} We shall prove the statement for the upper limits, starting with the first inequality. Given $n\in\NN$, let $F\subset X$ be such that  
$$\#F\,=\,Q_{\mu,W_{1,n}(\ep)}
\qquad\text{and}\qquad\int   d_n(x,F)\,d \mu(x)\,\,\leq\,\,\ep.$$
By Markov's inequality %(cf. \cite{Ross})
,
$$\mu\Big(d_n(\cdot\,,\,F)\,\geq\,2\,\ep \Big)\,\,\leq\,\,\frac{1}{2\,\ep}\,\int   d_n(x,F)\,d \mu(x)\,\,\leq\,\,\frac{1}{2}.$$ Thus,
$$\mu\Big(\bigcup_{y\,\in\,F} \, B_n(y,2\,\ep) \Big)
\,\,\geq\,\, 1/2\,\,>\,\,1\,-\,3/4$$
and so
$$Q_{\mu,W_{1,n}(\ep)}\,\,\geq\,\, \mathcal{N}_\mu(n,2\,\vep, 3/4).$$
We conclude the argument by applying the logarithm to both sides of this inequality, dividing by $n$ and taking $n \to +\infty.$ 
\smallskip

The second inequality is immediate since the map $p\,\mapsto\,W_p$ is non-decreasing (cf. \cite[p.~210]{Vi}) or \eqref{MetricRelations}). Regarding the third inequality, fix 
$$0 \,<\,\delta\,<\,\frac{1}{2}\,\Big(\frac{\ep}{\diam(X,d)}\Big)^p$$ and 
$$ s\,>\,\overline{h}^K_\mu(2^{-1/p}\ep, \delta).$$
Then, for sufficiently large $n\in\NN$, we have $$\cN_\mu(n,2^{-1/p}\ep, \delta)\,<\,e^{n\,s}.$$
Take $F\subset X$ such that 
$$\mu\Big( \bigcup_{y\,\in\,F} \, B_n(y,2^{-1/p}\,\ep) \Big)\,\,\,>\,\,\, 1-\delta\qquad\text{and}\qquad \#F\,<\,e^{n\,s}.$$ 
Thus, 
\begin{align*}
\int   d_n(\cdot\,,\,F)^p\,d \mu \,\,&\leq\,\,\int_{\displaystyle \bigcup_{y\,\in\,F} \, B_n(y,2^{-1/p}\,\ep)} \,\,  d_n(\cdot\,,\,F)^p\,d \mu\,\,+\,\,\int_{\displaystyle X \setminus \bigcup_{y\,\in\,F} \, B_n(y,2^{-1/p}\,\ep)}\,\,    d_n(\cdot\,,\,F)^p\,d \mu \\
&\leq \,\, \frac{\ep^p}{2}\,\,\mu\Big( \bigcup_{y\,\in\,F} \, B_n(y,2^{-1/p}\,\ep) \Big)\,\,+\,\, \delta \,\cdot\, \diam(X,d_n)^p\\ 
&\leq \,\, \frac{\ep^p}{2}\,\,+\,\, \delta \, \cdot \, \diam(X,d)^p\,\,\,\,\,\leq\,\,\,\,\, \ep^p.
\end{align*}
Thus, 
$$Q_{\mu,D_n}(\ep)\,\,\leq\,\,\#F\,\,<\,\, e^{ns}$$
and so 
$$\overline{Q}_\mu(T,W_p,\ep)\,\,\leq\,\,s.$$
To finish the proof, we just have to make  $s\,\to\,\overline{h}^K_\mu(2^{-1/p}\ep, \delta)$ and take the limit as $\delta\to 0^+$.
% The last equalities in the statement of Lemma~\ref{QuantGrowthRate} are a consequence of the fact that the functions $\overline{h}^K_\mu(2^{-1/p}\ep, \delta)$ and $\underline{h}^K_\mu(2^{-1/p}\ep, \delta)$  increase as $\delta$ decreases.
\end{proof}

By \eqref{Flimit}, \eqref{vpErcai} and the fact that 
$$\Big\{\,\overline{h}^K_\mu(2\ep,3/4),\,\, \overline{h}^K_\mu(\ep),\,\,\underline{h}^K_\mu(2\ep,3/4),\,\,
\underline{h}^K_\mu(\ep)\Big\} \,\subset \,\cD$$ 
it is easy to see that Lemma~\ref{QuantGrowthRate} yields:

\medskip

\noindent $(i)$ For every $\mu\,\in\,\cE_T(X)$, $F\,\in\,\cD$ and $p\,\in[1, +\infty)$,
\begin{eqnarray*}
\overline{\mathrm{mdim}}_Q(X,d,T,W_p,\mu) &\,=\,& \limsup_{\ep\,\to\,0^+}\,\frac{F(\mu,\ep)}{\log\,(1/\ep)} %\\ 
%\underline{\mathrm{mdim}}_Q(X,d,T,W_p,\mu) &\,=\,&\liminf_{\ep\,\to\,0^+}\,\frac{F(\mu,\ep)}{\log\,(1/\ep)}.
\end{eqnarray*}
and, for every $p\,\in[1, +\infty)$,
\begin{equation*}
    \begin{alignedat}{3}
        \overline{\mdim}_M(X,d,T) \, \,\,&=\,\,\, \limsup_{\varepsilon \, \to \, 0^+}\, \sup_{\mu \,\in\, \mathcal{E}_T(X)} \, \frac{\overline{Q}_\mu(T, W_p,\ep)}{\log\,(1/\varepsilon)} \, \,\,&=\,\,\,\, \limsup_{\varepsilon \, \to \, 0^+}\, \sup_{\mu \,\in\, \mathcal{E}_T(X)} \, \frac{\underline{Q}_\mu(T, W_p,\ep)}{\log\,(1/\varepsilon)}
        %\\
        %&=\,\, \,\limsup_{\varepsilon \, \to \, 0^+}\, \sup_{\mu \,\in\, \mathcal{P}_T(X)} \, \frac{\overline{Q}_\mu(T,W_p,\ep)}{\log\,(1/\varepsilon)} \,\, \,&=\,\,\, \limsup_{\varepsilon \, \to \, 0^+}\, \sup_{\mu \,\in\, \mathcal{P}_T(X)} \, \frac{\overline{Q}_\mu(T,W_p,\ep)}{\log\,(1/\varepsilon)}
        \mathrlap{.}
    \end{alignedat}
\end{equation*} 
\smallskip
Combining the previous equalities with \eqref{QuantLeqEnt}, we obtain:

\medskip

\noindent $(ii)$ For every $p\,\in[1, +\infty)$,
\begin{equation*}
    \begin{alignedat}{3}
        \overline{\mdim}_M(X,d,T) \, \,\,&=\,\,\, \limsup_{\varepsilon \, \to \, 0^+}\, \sup_{\mu \,\in\, \mathcal{E}_T(X)} \, \frac{\overline{Q}_\mu(T,W_p,\ep)}{\log\,(1/\varepsilon)} \, \,\,&=\,\,\,\, \limsup_{\varepsilon \, \to \, 0^+}\, \sup_{\mu \,\in\, \mathcal{E}_T(X)} \, \frac{\underline{Q}_\mu(T,W_p,\ep)}{\log\,(1/\varepsilon)}
        \\
        &=\,\, \,\limsup_{\varepsilon \, \to \, 0^+}\, \sup_{\mu \,\in\, \mathcal{P}_T(X)} \, \frac{\overline{Q}_\mu(T,W_p,\ep)}{\log\,(1/\varepsilon)} \,\, \,&=\,\,\, \limsup_{\varepsilon \, \to \, 0^+}\, \sup_{\mu \,\in\, \mathcal{P}_T(X)} \, \frac{\overline{Q}_\mu(T,W_p,\ep)}{\log\,(1/\varepsilon)}
        \mathrlap{.}
    \end{alignedat}
\end{equation*} 

\medskip

To finish the proof of Theorem~\ref{teo:QD}, we need further information.

\begin{lemma}\cite[Lemma 3.18]{BB}\label{def:MPQ}
For every $\mu, \nu \in \cP(X)$ such that $\mu \geq t\nu$ for some $t > 0$, one has
$$Q_{\mu,W_1}(t\ep)\, \geq\, Q_{\nu,W_1}(\ep) \quad \forall\, \vep>0.$$
\end{lemma}

\smallskip

Applying the previous lemma to $(X,d_n)$, we obtain \begin{eqnarray*}
    \overline{Q}_{\mu}(T,W_1,t\ep)\, &\geq&\, \overline{Q}_{\nu}(T,W_1,\ep)
    \\
     \underline{Q}_{\mu}(T,W_1,t\ep)\, &\geq&\, \underline{Q}_{\nu}(T,W_1,\ep)
\end{eqnarray*}

\noindent whenever $\mu \geq t\nu$ for some $t>0$. This means that the $W_1$-dynamical quantization numbers satisfy the s.m.m.d.~property.
%smmd.
\smallskip

Since the map $p\in[1, +\infty)\,\mapsto\,W_p$ is non-decreasing (see \eqref{MetricRelations}), combining Corollary~\ref{teo:VP} with \eqref{MQuantLeqMDim} we conclude that, for every $p\,\in[1, +\infty)$,
\begin{equation}
\begin{alignedat}{2}
\mathrm{\overline{mdim}_M}(X,d,T)\,\,\,\, 
&=\,\,\,\max_{\mu\,\in\,\cP_T(X)}\,\, \,\mathrm{\overline{mdim}_Q}(X,d,T,W_1,\mu) \\ 
& \leq\,\,\, \max_{\mu\,\in\,\cP_T(X)}\,\, \,\mathrm{\overline{mdim}_Q}(X,d,T,W_p,\mu) \\ 
& \leq\,\,\,\,\mathrm{\overline{mdim}_M}(X,d,T)
\mathrlap{.}
\end{alignedat}
\end{equation}

\smallskip

\subsection{$LP$-mean quantization dimension}

Definition~\ref{def:qnumber} admits a reformulation when $D$ is the L\'evy-Prokhorov metric.

\begin{lemma}\cite[Proposition 3.3]{BB} The quantization number $Q_{\mu,LP}(\ep)$ for the $LP$ metric is the least number of closed balls with radius $\vep$ that cover any subset of $X$ with $\mu$ measure at least $1-\vep$.
\end{lemma}

Observe that if we consider the space $(X,d_n)$, the above formulation resembles the cover number $\mathcal{N}_\mu(n,\vep, \delta)$ from the definition of the $(\ep,\delta)$-Katok entropy when $\delta=\ep$, the only difference being that the balls in the covers are closed. Therefore:

\begin{lemma}\label{le:QK}
For every $\mu\in \cP_T(X)$ and $\ep>0$, 
\begin{equation*}
    \begin{alignedat}{2}
\overline{h}^K_\mu(2\ep,\ep)\,\,\,\,&\leq\,\,\,\,
\overline{Q}_\mu(T,LP,\ep) \,\,\,\,&\leq\,\,\,\, \overline{h}^K_\mu(\ep,\ep) \\\underline{h}^K_\mu(2\ep,\ep)\,\,\,\,&\leq\,\,\,\,\underline{Q}_\mu(T,LP,\ep) \,\,\,\,&\leq\,\,\,\, \underline{h}^K_\mu(\ep,\ep) 
    \end{alignedat}.
\end{equation*}
\end{lemma}

In particular, for every $0<\ep<1/2$, we have
\begin{equation*}
    \begin{alignedat}{2}
\overline{h}^K_\mu(2\ep,1/2)\,\,\,\,&\leq\,\,\,\,\overline{Q}_\mu(T,LP,\ep) \,\,\,\,&\leq\,\,\,\, \overline{h}^K_\mu(\ep) \\\underline{h}^K_\mu(2\ep,1/2)\,\,\,\,&\leq\,\,\,\,\overline{Q}_\mu(T,LP,\ep) \,\,\,\,&\leq\,\,\,\, \underline{h}^K_\mu(\ep) 
    \end{alignedat}
\end{equation*}
and by a reasoning similar to that in the previous subsection, we prove items $(a)$ and $(b)$ of Theorem~\ref{teo:QD} when $D=LP$. 
\smallskip

To conclude, recall that, since the metric $LP$ is at least a uniform constant factor of $W_1$ (see~\eqref{MetricRelations}), we have
\begin{equation}
    \begin{alignedat}{2}
        \mathrm{\overline{mdim}_M}(X,d,T)\,\,\,\, &=\,\,\,\, \max_{\mu\,\in\,\cP_T(X)}\,\, \,\mathrm{\overline{mdim}_Q}(X,d,T,W_1,\mu) 
    \\ & \leq\,\,\,\, \max_{\mu\,\in\,\cP_T(X)}\,\, \,\mathrm{\overline{mdim}_Q}(X,d,T,LP,\mu) 
     \\ & \leq\,\,\,\,\mathrm{\overline{mdim}_M}(X,d,T)
    \mathrlap{.}
    \end{alignedat}
\end{equation}

The property $(d)$ of Theorem~\ref{teo:QD} is a straightforward consequence of the item $(c)$ of Theorem~\ref{teo:VPg}. The proof of Theorem~\ref{teo:QD} regarding the upper metric mean dimension is finished. A similar argument yields the corresponding statement for the lower metric mean dimension.

\section{Proof of Theorem~\ref{teo:EYZ}}\label{se:proofD}

Since the $(\ep,\delta)$-Katok entropy is non-increasing with respect to $\delta$, Lemma~\ref{le:QK} implies that, for every $\mu \in \cP_T(X)$ and $\ep>0$,
\begin{equation*}
    \begin{alignedat}{3}
\overline{h}^K_\mu(2\ep,2\ep)\,\,\,\,&\leq\,\,\,\,\overline{h}^K_\mu(2\ep,\ep)\,\,\,\,&\leq\,\,\,\,\overline{Q}_\mu(T,d,LP,\ep) \,\,\,\,&\leq\,\,\,\, \overline{h}^K_\mu(\ep,\ep) %\\\underline{h}^K_\mu(2\ep,2\ep)\,\,\,\,&\leq\,\,\,\,\underline{h}^K_\mu(2\ep,\ep)\,\,\,\,&\leq\,\,\,\,\underline{Q}_\mu(T,d,LP,\ep) \,\,\,\,&\leq\,\,\,\, \underline{h}^K_\mu(\ep,\ep) 
    \end{alignedat}.
\end{equation*}

\smallskip

\noindent Combining this information with Lemma~\ref{K=S}, we obtain, for every $\mu\in\cP_T(X)$,
\begin{eqnarray*}
\overline{\mathrm{mdim}}_Q(X,d,T,LP,\mu) &=& 
\limsup_{\ep\,\to\,0^+}\,\frac{\overline{h}^K_\mu(\ep,\ep)}{\log\,(1/\ep)} 
\,\,\,= \,\,\, 
\limsup_{\ep\,\to\,0^+}\,\frac{\overline{h}^S_\mu(\ep,\ep)}{\log\,(1/\ep)} 
%\\ 
%\underline{\mathrm{mdim}}_Q(X,d,T,LP,\mu) &=&
%\liminf_{\ep\,\to\,0^+}\,\frac{\overline{h}^K_\mu(\ep,\ep)}{\log\,(1/\ep)}\,\,\,
%\,\,\,= \,\,\, \liminf_{\ep\,\to\,0^+}\,\frac{\overline{h}^S_\mu(\ep,\ep)}{\log\,(1/\ep)}.
\end{eqnarray*}

\medskip

\begin{remark}\label{rem:smmd}
By definition, for a fixed $\delta>0$, the $(\ep,\delta)$-Katok entropy satisfies $$\overline{h}_\mu^K(\ep,\delta)\,\,\geq\,\,\overline{h}_\nu^K(\ep,\delta/t)\,\,\geq\,\,\overline{h}_\nu^K(\ep/t,\delta/t)$$ 
whenever $\mu \geq t\,\nu$ for some $t\in(0,1)$. In particular, $\overline{h}_\mu^K(\ep,\ep)$ has the s.m.m.d.~property, and the map $\overline{h}_\mu^K(\ep)\,=\,\lim_{\delta\,\to\,0^+}\,\overline{h}_\mu^K(\ep,\delta)$ satisfies the following stronger version of that property:
$$\overline{h}_\mu^K(\ep)\,\,\geq\,\,\overline{h}_\nu^K(\ep)\,\,\geq\,\,\overline{h}_\nu^K(\ep/t).$$
The corresponding statements for the Shapira entropies also hold.
\end{remark}

%%%%%%%%%%%%%%%%%%%%%%%%%%%%%%%%%%%%%%%%%%%%%%%%%%%%%%%%%%%%
\subsection{Example}\label{se:exampleB}

Consider the interval $[0,1]$ with the Euclidean distance, the space $[0,1]^\ZZ$ endowed with the following distance, which is compatible with the product topology,
$$d\big((x_n)_{n\,\in\,\ZZ},(y_n)_{n\,\in\,\ZZ}\big)\,\,=\,\,\sum_{n\,\in\,\ZZ}\,\frac{|x_n\,-\,y_n|}{2^{|n|}}$$ and 
%It is known (cf. \cite{LW2000}) that 
the shift map $\si\big((x_n)_{n\,\in\,\ZZ}\big)=(x_{n+1})_{n\,\in\,\ZZ}.$
%satisfies $$\mdim_M([0,1]^\ZZ,d,\si)\,\,=\,\,1.$$
We refer to a closed shift-invariant subset $Y\subset [0,1]^\ZZ$ as a \emph{subshift}.

\medskip

Fix a sequence $a\,=\,(a_k)_{k\,\geq\,0}$ of real numbers and a sequence $b\,=\,(b_k)_{k\,\geq\,1}$ of positive integers that meet the following conditions:
\begin{equation*}\label{conditionshorseshoe}
    \begin{split}
        & \text{(C1) }\hspace{.3cm} 0=a_0<a_1<...<a_k\to1.\hspace{8cm} \\ & \text{(C2)}\hspace{.3cm} (a_{k}-a_{k-1})_k \text{ decreases to zero.}\\ & \text{(C3) }\hspace{.3cm} (b_k)_k \text{ is strictly increasing.}
    \end{split}
\end{equation*} 
To simplify the notation, we write 
\begin{equation}\label{ep_k}
    \ep_k \,\,=\,\, \frac{|a_k\,-\,a_{k-1}|}{b_k}
\end{equation}
and, for each $k\in\NN,$
%$$a_k^{(j)}\,\,=\,\,a_{k-1}\,\,+\,\,j\,\ep_k\qquad\quad\forall\, j\,\in\,\{0,...,b_{k}-1\}.$$
$$J_k\,\,=\,\,\Big\{a_{k-1}\,+\,j\,\ep_k\,\,\,\colon\,\, j\,\in\,\big\{0,...,b_{k}-1\big\}\Big\}.$$ Then, $$J_i\,\cap\,J_j\,=\,\emptyset\,\,\,\,\,\,\forall\, i\ne j\qquad \text{and}\qquad |x-y|\,\geq\,\ep_k\,\,\,\,\,\,\forall\,x,y\,\in\,J_k.$$ Denote $X_k\,=\,\big(J_k\big)^\ZZ$ and the subshift
\begin{equation*}
X\,\,=\,\,\bigcup_{k=1}^{+\infty}\,X_k\,\,\cup\,\,\big\{1^\ZZ\big\}\,\,\subset\,\,[0,1]^\ZZ.
\end{equation*}

By construction, every ergodic measure of the system $(X,d,\si)$ is either the Dirac mass $\delta_{1^\ZZ}$ or gives full mass to $X_k$, for some $k\in\NN$. Thus, by the ergodic decomposition theorem, we may write any invariant probability measure $\mu\,\in\,\cP_\si(X)$ as 
\begin{equation}\label{MeasureDecompositionCaixinhas}
\mu\,\,=\,\,\sum_{k\,\in\,\NN}\, t_k\,\mu_k\,\,+\,\,t_0\,\delta_{\{1^\ZZ\}},
\end{equation}
where $t_0\,=\,\mu(\{1^\ZZ\})$ and 
$$\mu_k\,\in\,\cP_\si(X),\qquad \mu_k(X_k)\,=\,1,\qquad t_k\,=\, \mu(X_k)\qquad\quad\forall\,k\in\NN.$$  

\noindent We recall that, given $k\in\NN$, the topological entropy of the shift map restricted to $X_k$ is $\log b_k$. So, the topological entropy of $(X,\sigma)$ is infinite and there exists an invariant probability measure with maximal entropy, but it is not ergodic.
\smallskip

In what follows, we assume that the parameters are chosen so that 
$$\mathrm{mdim}_M(X,d,\si)\,=\,1 \quad \quad \text{and} \quad \quad 
\lim_{k\,\to\,+\infty} \,\,\frac{\log \ep_k}{\log \ep_{k+1}}\,\,=\,\,1.$$
This is the case, for example, of $a_k\,=\,\sum_{n=1}^k 6/\pi n^2$ and $b_k\,=\,3^k$.

\subsection{Brin-Katok and Rényi entropies do not satisfy the variational principle \eqref{eq:vpgeral}}\label{BKcounterex}

We start by relating the behavior of convex entropy-like functions and the fixed point $1^\ZZ$.

\begin{proposition}\label{prop:conv}
Consider the system $(X,d,\si)$ and a map 
$$F\colon (\mu, \vep) \,\in \, \cP_\sigma(X) \times (0,+\infty) \mapsto F(\mu,\ep).$$
Assume that:
\begin{itemize}
\item[$(i)$] The map 
$$\mu \,\in \, \cP_\sigma(X) \,\longmapsto\,\, \limsup_{\ep\,\to\,0^+}\,\frac{F(\mu,\ep)}{\log(1/\ep)}$$ 
is convex.
\smallskip

\item[$(ii)$]  For every $\vep > 0$ and $\mu \in \cP_\sigma(X)$, 
$$F(\mu,\ep)\,\,\leq\,\,h_\ep(\supp(\mu),\,d|_{\supp(\mu)},\,\si|_{\supp(\mu)}).$$
\end{itemize}

\medskip

\noindent Then 
\begin{equation}\label{eq:convex}
\limsup_{\ep\,\to\,0^+}\,\frac{F(\mu,\ep)}{\log(1/\ep)}\,\,\leq\,\,\mu(\{1^\ZZ\})\,\,\limsup_{\ep\,\to\,0^+}\,\frac{F(\delta_{\{1^\ZZ\}},\ep)}{\log(1/\ep)}\qquad \forall\, \mu\in\cP_\si(X).    
\end{equation}
\end{proposition}
\smallskip

\begin{proof}
Take $\mu \in \cP_\sigma(X)$ and consider its decomposition  \eqref{MeasureDecompositionCaixinhas}, say
$$\mu\,\,=\,\,\sum_{k\,\in\,\NN}\, t_k\,\mu_k\,\,+\,\,t_0\,\delta_{\{1^\ZZ\}}.$$
Fix $K\in\NN$. Then
\begin{align*}
\limsup_{\ep\,\to\,0^+}\,\frac{F(\mu,\ep)}{\log(1/\ep)}\,\,&\leq\,\,\sum_{k\,=\,1}^K\, t_k\,\limsup_{\ep\,\to\,0^+}\,\frac{F(\mu_k,\ep)}{\log(1/\ep)}\,\,+\,\,\sum_{k\,>\,K}\, t_k\,\limsup_{\ep\,\to\,0^+}\,\frac{F(\mu_k,\ep)}{\log(1/\ep)} \\
&\qquad\,+\,\,t_0\,\limsup_{\ep\,\to\,0^+}\,\frac{F(\delta_{\{1^\ZZ\}},\ep)}{\log(1/\ep)}
%\\&\leq\,\,\sum_{k\,=\,1}^K\, t_k\,\limsup_{\ep\,\to\,0^+}\,\frac{h_{\ep}(X_k,d|_{X_k},\si|_{X_k})}{\log(1/\ep)}\,\,+\,\,\sum_{k\,>\,K}\, t_k\,\limsup_{\ep\,\to\,0^+}\,\frac{h_\ep(X,d,\si)}{\log(1/\ep)}\,\,+\,\,t_0\,\limsup_{\ep\,\to\,0^+}\,\frac{F(\delta_{\{1^\ZZ\}},\ep)}{\log(1/\ep)}
\\&\leq\,\,\sum_{k\,=\,1}^K\, t_k\,\limsup_{\ep\,\to\,0^+}\,\frac{\log b_K}{\log(1/\ep)}\,\,+\,\,\sum_{k\,>\,K}\, t_k\,\limsup_{\ep\,\to\,0^+}\,\frac{h_\ep(X,d,\si)}{\log(1/\ep)} \\
&\qquad\,+\,\,\,t_0\,\limsup_{\ep\,\to\,0^+}\,\frac{F(\delta_{\{1^\ZZ\}},\ep)}{\log(1/\ep)}\\
&=\,\,\mu\big( \bigcup_ {k\,>\,K} J_k\big)\,\mathrm{mdim}_M(X,d,\si)\,\,+\,\,t_0\,\limsup_{\ep\,\to\,0^+}\,\frac{F(\delta_{\{1^\ZZ\}},\ep)}{\log(1/\ep)}.
\end{align*}
Making $K\to+\infty$, we get \eqref{eq:convex} since
$\lim_{K \, \to \, +\infty}\,\mu\big( \bigcup_ {k\,>\,K} J_k\big) \, = \, 0.$
\end{proof}

Taking into account that
$$\limsup_{\ep\,\to\,0^+}\,\frac{\overline{h}_{\delta_{\{1^\ZZ\}}}^{BK}(\ep)}{\log(1/\ep)} \, = \, 0$$
\smallskip

\noindent and using the properties of $\overline{h}_\mu^{BK}$ described in Lemmas~\ref{BK_leq_ep-entropy} and \ref{BK_convex}, Proposition~\ref{prop:conv} yields:

\begin{corollary}\label{cor:BK}
For the system $(X,d,\si)$ we have $$\limsup_{\ep\,\to\,0^+}\,\frac{\overline{h}_\mu^{BK}(\ep)}{\log(1/\ep)}\,\,=\,\,0\qquad \forall\, \mu\in\cP_\si(X).$$
Therefore, as $\mathrm{\overline{mdim}_M}(X,d,\sigma)=1$,
$$\mathrm{\overline{mdim}_M}(X,d,\sigma) \,\, >\, \sup_{\mu\,\in\,\cP_\sigma(X)}\,\,  \limsup_{\vep \, \to \, 0^+} \, \,\frac{\overline{h}_\mu^{BK}(\ep)}{\log \,(1/\vep)}.$$
\end{corollary}

\smallskip

We have a similar statement for the Rényi entropy.

\begin{corollary}
For every $\mu\in\cP_\si(X)$, $$\limsup_{\ep\,\to\,0^+}\,\frac{\inf_{\mathrm{diam}(P)\,\le\, \varepsilon} \,\,h_\mu(P)}{\log(1/\ep)}\,\,=\,\,0.$$
Therefore, as $\mathrm{\overline{mdim}_M}(X,d,\sigma)=1$,
$$\mathrm{\overline{mdim}_M}(X,d,\sigma) \,\, >\, \sup_{\mu\,\in\,\cP_\sigma(X)}\,\,  \limsup_{\vep \, \to \, 0^+} \, \,\frac{\inf_{\mathrm{diam}(P)\,\le\, \varepsilon}\,\, h_\mu(P)}{\log \,(1/\vep)}.$$
\end{corollary}

\begin{proof}
We are left to show that the Rényi entropy satisfies the assumptions of Proposition~\ref{prop:conv}. We start by proving that the map
\begin{equation}\label{eq:map}
\mu\,\in\,\cP_\si(X)\,\longmapsto\,\,\limsup_{\vep \, \to \, 0^+} \, \,\frac{\inf_{\mathrm{diam}(P)\,\le\, \varepsilon} \,\,h_\mu(P)}{\log \,(1/\vep)}
\end{equation}
is convex. The next property was first established for ergodic measures in \cite{GS} and then extended to the invariant probability measures in \cite{Yang}.\\

\begin{claim}\cite{GS, Yang} \emph{For every $\mu\,\in\,\cP_\si\big([0,1]^\ZZ\big)$, 
\begin{equation}\label{eq:Pm}
\limsup_{\vep \, \to \, 0^+} \, \frac{\inf_{\mathrm{diam}(P)\,\le\, \varepsilon}\,\, h_\mu(P)}{\log \,(1/\vep)}\,\,=\,\,\limsup_{m \, \to \, +\infty} \, \frac{ h_\mu(P_m)}{\log \,m}
\end{equation}
where $m\in\NN$,
$$P_m \,=\, \pi_0^{-1}\Big\{\Big[\frac{i}{m},\frac{i+1}{m}\Big)\,\colon\,i=0,...,m-1 \Big\}$$
and $\pi_0 \colon [0,1]^\ZZ \to [0,1]$ is the projection on the $0$-coordinate.} 
\end{claim}
\medskip

In particular, the partitions considered on the limit on the right-hand side of \eqref{eq:Pm} do not depend on the measure. Since, for each $m \in \NN$, the map $\mu\mapsto h_\mu(P_m)$ is affine (cf. \cite[Proof of Theorem 8.1]{Wa}), and the $\limsup$ of affine functions is convex, we conclude that the map \eqref{eq:map} is convex.
\smallskip

The second assumption of Proposition~\ref{prop:conv} was established in \cite[Section~3]{GS}.
%Ver equação (3.4).
\end{proof}

%%%%%%%%%%%%%%%%%%%%%%%%%%%%%%%%%%%%%%%%%%%%%%%%%%%%%

\subsection{$(\ep,\delta)-$Katok/Shapira entropy with fixed $\delta$}\label{constantdelta}

Fix $\delta_0\in (0,1)$. Given $\mu\in\cP_\si(X)$, there exists $K=K(\mu,\delta_0)$ such that $\mu\big( \bigcup_ {k\,>\,K} J_k\big) < \delta_0$. Since $\si$ restricted to $\bigcup_ {k\,>\,K} J_k\cup\{1^\ZZ\}$ has a finite topological entropy, one has
$$\limsup_{\ep\,\to\,0^+}\,  \frac{\overline{h}^K_{\mu}(\ep,\delta_0)}{\log\,(1/\ep)}\,\,=\,\,0.$$
Therefore, since this equality is valid for every $\mu\in\cP_\si(X)$, we conclude that 
$$\mathrm{\overline{mdim}_M}(X,d,T) \,\,> \,\, \sup_{\mu\,\in\,\cP_T(X)}\,\limsup_{\ep\,\to\,0^+}\, \frac{\overline{h}^K_{\mu}(\ep,\delta_0)}{\log\,(1/\ep)} .$$ By Lemma~\ref{K=S}, the above inequality also holds for the $(\ep,\delta_0)$-Shapira entropy.

%%%%%%%%%%%%%%%%%%%%%%%%%%%%%%%%%%%%%%%%%%%%%%%%%%%%%%%%%%%%%%%%%%

\subsection*{Acknowledgments}
Gustavo Pessil is grateful to Prof. Yonatan Gutman for his mentoring at IMPAN. The authors thank Adam \'Spiewak, Chunlin Liu and Rui Yang for insightful comments on this work. Maria Carvalho was partially supported by CMUP, member of LASI, which is financed by national funds through FCT – Fundação para a Ciência e a Tecnologia, I.P., under the project UID/00144/2025, https://doi.org/10.54499/UID/00144/2025.

%%%%%%%%%%%%%%%%%%%%%%%%%%%%%%%%%%%%%%%%%%%%%%%%%%%%%%%%%%%%%%%%%%


\begin{thebibliography}{99}

%\bibitem{AR}
%J. Acevedo and F. Rodrigues.
%\newblock Mean dimension and metric mean dimension for non-autonomous dynamical systems.
%\newblock  J. Dyn. Control Syst. 28 (2022), 697--723. 

%\bibitem{Bis2013}
%A. Bi\'s.
%\newblock An analogue of the variational principle for group and pseudogroup actions.
%\newblock Ann. Inst. Fourier 63:3 (2013) 839--863.

%\bibitem{BCP}
%A. Baraviera, M. Carvalho and G. Pessil.
%\newblock \emph{Metric mean dimension, H\"older regularity and Assouad spectrum.}
%\newblock J. Fractal Geom. (2025), published online first.

%\bibitem{BDGS}
%A. Bi\'s, D. Dikranjan, A. Giordano Bruno and L. Stoaynov.
%\newblock Topological entropy, upper Carath\'eodory capacity and fractal dimensions of semigroup actions.
%\newblock Colloq. Math. 163:1 (2021) 131--151.


\bibitem{BB}
P. Berger and J. Bochi.
\newblock \emph{On emergence and complexity of ergodic decompositions.}
\newblock Adv. Math. 390 (2021), 107904.

%\bibitem{BCMV}
%A. Bi\'s, M. Carvalho, M. Mendes and P. Varandas.
%\newblock A convex analysis approach to entropy functions, variational principles and equilibrium states.
%\newblock Commun. Math. Phys 394 (2022), 215--256.

%\bibitem{BCMV-2}
%A. Bi\'s, M. Carvalho, M. Mendes and P. Varandas.
%\newblock Entropy functions for semigroup actions.
%\newblock Preprint, 2022.

%\bibitem{Bil}
%P. Billingsley.
%\newblock  \emph{Convergence of Probability Measures.}
%\newblock Wiley Series in Probability and Statistics, John Wiley and Sons, New York, 2nd edition, 1999.

%\bibitem{BrinK}
%M. Brin and A. Katok.
%\newblock \emph{On local entropy.}
%\newblock Lecture Notes in Math. 1007, Springer Berlin, 1983, pp. 30--38.


%\bibitem{BrinS}
%M. Brin and G. Stuck.
%\newblock \emph{Introduction to Dynamical Systems.}
%\newblock Cambridge University Press, 2002.

%\bibitem{Bufetov}
%A. Bufetov.
%\newblock Topological entropy of free semigroup actions and skew-product transformations.
%\newblock J. Dyn. Control Sys. 5:1 (1999) 137--142.

%\bibitem{Bre}
%J. Br\'emont.
%\newblock Entropy and maximizing measures of generic continuous functions.
%\newblock C. R. Math. Acad. Sci. S\'er. I 346 (2008), 199--201.

%\bibitem{CPV}
%M. Carvalho, G. Pessil and P. Varandas.
%\newblock \emph{A convex analysis approach to the metric mean dimension: limits of scaled pressures and variational principles.}
%\newblock Adv. Math. 436 (2024) 109407.

%\bibitem{CRV}
%M. Carvalho, F. B. Rodrigues and P. Varandas.
%\newblock A variational principle for free semigroup actions.
%\newblock Adv. Math. 334 (2018), 450--487.

%\bibitem{CRV4}
%M. Carvalho, F. Rodrigues and P. Varandas.
%\newblock A variational formula for the metric mean dimension of free semigroup actions.
%\newblock Ergodic Theory Dynam. Systems 42 (2021), 65--85.

%\bibitem{CCL}
%H. Chen, D. Cheng and Z. Li.
%\newblock Upper metric mean dimensions with potential.
%\newblock Results Math. 77 (2022), 54. 


\bibitem{EYZ2025}
E. Chen, R. Yang and X. Zhou.
\newblock \emph{Measure-theoretic metric mean dimension.}
\newblock Studia Math. 280(1) (2025), 1--25.

%\bibitem{CL}
%D. Cheng and Z. Li.
%\newblock Scaled pressure of dynamical systems.
%\newblock J. Differential Equations 342 (2023), 441--471.

%\bibitem{Chung}
%N.-P. Chung.
%\newblock Topological pressure and the variational principle for actions of sofic groups.
%\newblock Ergodic Theory Dyn. Syst. 33:5 (2012) 1363--1390.
%
%\bibitem{GLW}
%E. Ghys, R. Langevin and P. Walczak.
%\newblock Entropie geometrique des feuilletages.
%\newblock Acta Math. 16 (1988) 105--142.

%\bibitem{Con}
%G. Contreras.
%\newblock Ground states are generically a periodic orbit.
%\newblock Invent. Math. 205:2 (2016), 383--412.

%\bibitem{FH16}
%D.-J. Feng and W. Huang.
%\newblock Variational principle for weighted topological pressure.
%\newblock J. Math. Pures Appl. 106 (2016), 411--452.

\bibitem{GSu}
A. L. Gibbs and F. E. Su.
\newblock \emph{On choosing and bounding probability metrics.}
\newblock Int. Stat. Rev. 70(3) (2002), 419--435.


\bibitem{GL}
S. Graf and H. Luschgy.
\newblock \emph{Foundations of Quantization for Probability Distributions.}
\newblock Lecture Notes in Mathematics 1730, Springer-Verlag, Berlin, 2000.

%\bibitem{Gr1999}
%M. Gromov.
%\newblock \emph{Topological invariants of dynamical systems and spaces of holomorphic maps: I.}
%\newblock Math. Phys. Anal. Geom. 2 (1999), 323--415.

\bibitem{GS}
Y. Gutman and A. \'Spiewak.
\newblock \emph{Around the variational principle for metric mean dimension.}
\newblock Studia Math. 261(3) (2021), 345--360.

%\bibitem{Jenk}
%O. Jenkinson.
%\newblock Survey: Ergodic optimization in dynamical systems.
%\newblock Ergodic Theory Dynam. Systems 39:10 (2019), 2593--2618. 

%\bibitem{KD94}
%T. Kawabata and A. Dembo.
%\newblock \emph{The rate-distortion dimension of sets and measures.}
%\newblock IEEE Trans. Inf. Theory 40(5) (1994), 1564--1572.

%\bibitem{Katok}
%A. Katok.
%\newblock \emph{Lyapunov exponents, entropy and periodic orbits for diffeomorphisms.}
%\newblock Publ. Math. Inst. Hautes  \'Etudes Sci. 51 (1980), 137-%-173.

%\bibitem{Kolmogorov-Tikhomirov}
%A. N. Kolmogorov and V. M. Tikhomirov.
%\newblock \emph{$\epsilon$-Entropy and $\epsilon$-capacity of sets in %functional spaces.}
%\newblock Amer. Math. Soc. Transl. 2:17 (1961) 277--364.

\bibitem{LR24}
G. Lacerda and S. Romaña.
\newblock \emph{Mean topological dimension.}
\newblock IEEE Trans. Inform. Theory 70(11) (2024), 7664--7672.


%\bibitem{Lind1999}
%E.~Lindenstrauss.
%\newblock Mean dimension, small entropy factors and an embedding theorem.
%\newblock Publ. Math. IHES 89:1 (1999) 227--262.

\bibitem{LW2000}
E. Lindenstrauss and B. Weiss.
\newblock \emph{Mean topological dimension.}
\newblock Israel J. Math. 115 (2000), 1--24.

\bibitem{LT18}
E. Lindenstrauss and M. Tsukamoto.
\newblock \emph{From rate distortion theory to metric mean dimension: variational principle.}
\newblock IEEE Trans. Inform. Theory 64 (2018), 3590--3609.

\bibitem{LT19}
E. Lindenstrauss and M. Tsukamoto.
\newblock \emph{Double variational principle for mean dimension.}
\newblock Geom. Funct. Anal. 29 (2019), 1048--1109.

%\bibitem{MM}
%M. Misiurewicz.
%\newblock A short proof of the variational principle for a $\mathbb Z^n_+$-action on a compact space.
%\newblock Bull. Acad. Polon. Sci. S\'er. Sci. Math. Astronom. Phys. 24:12 (1976) 1069--1075.

%\bibitem{Morr}
%I. Morris.
%\newblock Ergodic optimization for generic continuous functions.
%\newblock Discrete Contin. Dyn. Sys. 27 (2010), 383--388.

%\bibitem{OP}
%J. Ollagnier and D. Pinchon.
%\newblock The variational principle.
%\newblock Studia Math. 72 (1982) 151--159.

\bibitem{OU}
J. C. Oxtoby and S. M. Ulam.
\newblock \emph{Measure-preserving homeomorphisms and metrical transitivity.}
\newblock Ann. Math. 42(4) (1941), 874--920.

%\bibitem{Ross}
%S. Ross.
%\newblock \emph{A First Course in Probability.}
%\newblock Pearson Prentice Hall, 2006.

%\bibitem{Ru}
%D. Ruelle.
%\newblock Statistical mechanics on a compact set with $Z^d$ action satisfying expansiveness and specification.
%\newblock Trans. Amer. Math. Soc. 187 (1973) 237--251.

\bibitem{Shapira}
U. Shapira.
\newblock \emph{Measure theoretical entropy of covers.}
\newblock Israel. J. Math. 158(1) (2007), 225--247.

\bibitem{SHI}
R. Shi.
\newblock \emph{On variational principles for metric mean dimension.}
\newblock IEEE Trans. Inform. Theory 68 (2022), 4282--4288.

\bibitem{Shi23}
R. Shi.
\newblock \emph{Finite mean dimension and marker property.}
\newblock Trans. Amer. Math. Soc. 376 (2023), 6123--6139.


%\bibitem{Shinoda}
%M. Shinoda.
%\newblock Uncountably many maximizing measures for a dense subset of continuous functions.
%\newblock Nonlinearity 31 (2018), 2192--2200.

%\bibitem{S}
%K. Sigmund.
%\newblock On dynamical systems with the specification property.
%\newblock  Trans. Amer. Math. Soc. 190 (1974), 285--299.

%\bibitem{Si}
%K. Sigmund.
%\newblock On the connectedness of ergodic systems.
%\newblock Manuscr. Math. 22 (1977), 27--32.

%\bibitem{Tsu2020}
%M. Tsukamoto.
%\newblock Double variational principle for mean dimension with potential.
%\newblock Adv. Math. 361 (2020), 106935.

\bibitem{Tsu2025}
M. Tsukamoto.
\newblock \emph{Rate distortion dimension and ergodic decomposition for $\mathbb{R}^d$-actions.}
\newblock Preprint arXiv:2503.06851, 2025. 

\bibitem{TTY}
M. Tsukamoto, M. Tsutaya and M. Yoshinaga.
\newblock \emph{G-index, topological dynamics and the marker property.}
\newblock Israel J. Math. 251 (2022) 737--764.

\bibitem{Vi}
C. Villani.
\newblock \emph{Topics in Optimal Transportation.}
\newblock Graduate Studies in Mathematics 58, American Mathematical Society, Providence, RI, 2003.

\bibitem{VV}
A. Velozo and R. Velozo.
\newblock\emph{Rate distortion theory, metric mean dimension and measure theoretic entropy.}	
\newblock Preprint, arXiv:1707.05762, 2017.

%\bibitem{XM}
%Q. Xiao and D. Ma.
%\newblock Topological pressure of free semigroup actions for non-compact sets and Bowen's equation.
%\newblock  J. Dynam. Differential Equations (2021), 

\bibitem{Wang}
T. Wang.
\newblock \emph{Variational relations for metric mean dimension and rate distortion dimension.}
\newblock Discrete Contin. Dyn. Syst. 41(10) (2021), 4593--4608.

\bibitem{Wa}
P. Walters.
\newblock \emph{An Introduction to Ergodic Theory.}
\newblock Graduate Texts in Mathematics 79, Springer-verlag, New York, Berlin, Heidelberg, 1982.

\bibitem{Yang}
R. Yang.
\newblock\emph{Mean dimension and rate-distortion function revisited.}	
\newblock Preprint, arXiv:2510.08051, 2025.

%\bibitem{YCZ}
%R. Yang, E. Chen and X. Zhou.
%\newblock On variational principle for upper metric mean dimension %with potential.
%\newblock Preprint 2022, arXiv: 2207.01901.

%\bibitem{YZ}
%X. Ye and G. Zhang.
%\newblock Entropy points and applications.
%\newblock Trans. Amer. Math. Soc. 359:12 (2007), 6167--6186.

\end{thebibliography}
\end{document}